\documentclass[12pt]{article}
\usepackage{a4wide}

\usepackage{latexsym,amssymb,amsmath,epsfig,amsfonts,algorithm}
\usepackage{mathtools}
\usepackage[title]{appendix}
\usepackage{eucal}
\usepackage{multirow,multicol,array,bm}
\usepackage[authoryear]{natbib}
\usepackage[normalem]{ulem}
\usepackage{tabularx, booktabs, makecell, caption}
\usepackage{siunitx}
\usepackage[colorlinks=true,allcolors=blue]{hyperref}
\usepackage{float}
\usepackage{caption}
\usepackage[utf8]{inputenc}
\usepackage[english]{babel}
\usepackage[noend]{algpseudocode}
\usepackage{tikz}
\usetikzlibrary{arrows}
\usepackage{bbm}
\usepackage[noend]{algpseudocode}
\usepackage{graphicx}
\usepackage[skip=0.333\baselineskip]{subcaption}
\usepackage[font=normal,labelfont=bf]{caption}
\usepackage{comment}
\usepackage{algpseudocode}
\algdef{SE}[DOWHILE]{Do}{doWhile}{\algorithmicdo}[1]{\algorithmicwhile\ #1}%

\DeclareMathOperator{\EX}{\mathbb{E}} 

\providecommand{\keywords}[1]{\textbf{\textit{Keywords---}} #1}
\numberwithin{table}{section}
\numberwithin{equation}{section}
\usepackage{amsthm}

\usepackage{color}

\newtheorem{theorem}{Theorem}

\newtheorem{proposition}[theorem]{Proposition}

\theoremstyle{remark}
\newtheorem{remark}{Remark}

\newcommand{\footremember}[2]{%
	\footnote{#2}
	\newcounter{#1}
	\setcounter{#1}{\value{footnote}}%
}

\title{Estimating customer delay and tardiness sensitivity from periodic queue length observations}

\author{%
	Liron Ravner\footremember{alley}{Department of Statistics, University of Haifa, 3491077 Haifa, Israel. lravner@stat.haifa.ac.il}%
	\and Jiesen Wang  \footremember{trailer}{School of Mathematics and Statistics, The University of Melbourne, 3010 Victoria, Australia.} \footnote{School of Mathematical Sciences, Tel Aviv University. jiesenwang@gmail.com}
}

\date{}

\begin{document}
\maketitle

\begin{abstract} 
	A single server commences its service at time zero every day. A random number of customers decide when to arrive to the system so as to minimize the waiting time and tardiness costs. The costs are proportional to the waiting time and the tardiness with rates $\alpha$ and $\beta$, respectively. Each customer's optimal arrival time depends on the others' decisions, thus the resulting strategy is a Nash equilibrium.  This work considers the estimation of the ratio $\displaystyle\theta\equiv \beta/(\alpha+\beta)$ from queue length data observed daily at discrete time points, given that customers use a Nash equilibrium arrival strategy. A method of moments estimator is constructed from the equilibrium conditions. Remarkably, the method does not require estimation of the Nash equilibrium arrival strategy itself, or even an accurate estimate of its support. The estimator is strongly consistent and the estimation error is asymptotically normal. Moreover,  the asymptotic variance of the estimation error as a function of the queue length covariance matrix (at sampling times) is derived.  The estimator performance is demonstrated through simulations, and is shown to be robust to the number of sampling instants each day. 
\end{abstract}


\section{Introduction}

When the quality of a product or a service deteriorates with time, an important consideration of customers is to avoid service delay. It occurs quite often that customers arrive at a supermarket earlier in order to obtain fresh products. An expensive managing warehouse always expects items to be picked up as soon as possible (see \citet{ELKS93}). The earlier a customer gets in a concert, the better seat she can take (see  \citet{JJS11}). In these cases, tardiness, measured from the opening time to the service commencement time of a customer, needs to be incorporated as a measure of the cost. 

{Customers typically try to avoid congestion.  In particular, besides not being late, a customer hopes to wait for as little time as possible.} Assume that the cost is linear in time with rates $\alpha$ and $\beta$ for the waiting and tardiness, respectively. Then in deciding when to arrive, a customer needs to balance between arriving earlier and waiting less, and the optimal decision is affected by the ratio $\beta/(\alpha+\beta)$. While the literature on strategic arrivals to queueing systems with tardiness costs provides considerable methods to describe and obtain the Nash equilibrium (e.g., \cite{H13,JS13}), it assumes that the system parameters are known beforehand. The problem of estimation concerning the cost rates has received little attention. In this paper, we estimate the ratio $\beta/(\alpha+\beta)$ from the queue length information.

We consider a single server queueing system with a first-come first-served (FCFS) discipline, which commences its service at time zero every day. The service time is exponentially distributed with rate $\mu$. The system can have a pre-imposed closing time or not. Let $T$ be the service closing time, then the system serves all the customers who arrive prior to time $T$ (inclusive). When $T = \infty$, the system closes after it serves all arriving customers. Every day, a Poisson number, with mean $\lambda$, of customers each decide when to arrive to this system,  so as to minimize their expected waiting and tardiness cost. When $T = \infty$, as $\lambda$ is finite, the system still serves all the arriving customers in a finite time. Thus, we still call it a day.

The FCFS queueing discipline implies that the expected cost of a customer depends on the number of customers who arrived before her, which is affected by the other customers' decisions. As a consequence, the best response of each customer is a function of the other customers' actions. {It is thus appropriate to model this interaction as a non-cooperative game with the solution concept of a Nash equilibrium strategy. We focus on the symmetric mixed Nash equilibrium,  which dictates that all customers have the same strategy, which is a probability distribution on $(-\infty, T]$.}

Our general framework allows for two model variations in which customers queueing before time zero is possible, or not. As in \citet{H13}, we refer to these variations as {\it with early birds} and {\it without early birds}. In the with early birds case, customers arriving before time zero are served according to a FCFS discipline.  In the without early birds case, customers who arrive before time zero are served in a random order. Hence, arriving before time zero does not bring extra benefit but incurs additional waiting costs, compared with arriving at time zero. So in the Nash equilibrium, customers do not arrive 
at any time $t<0$.

We assume that customers arrive according to the Nash equilibrium every day, and their decisions are independent on different days. The values of $\alpha$ and $\beta$ are known to the customers but not to the system manager.  To estimate the ratio $\theta=\beta/(\alpha+\beta)$, the system manager can sample the queue length at a collection of time instants every day. Note that although the observations from different days are independent, the observations from different time points in one day are not. We assume that the sampling instants are the same on all days. An estimator for $\theta$ is then derived by using the property that in Nash equilibrium the expected cost is equal at any time with a positive arrival rate.

\subsection{Literature review}

\citet{H13} derived the Nash equilibrium arrival distribution for the system that this work focuses on. The study of the strategic arrival time choice in queueing models goes back to \citet{GH83}, which studied a single server system with a Poisson number of customers that decide when to arrive with the goal of minimizing their expected waiting time. The Nash equilibrium arrival distribution is characterized as a solution of a system of functional differential equations that satisfy a condition of constant cost throughout the arrival support. \citet{JJS11} considered a concert queueing
game with waiting and tardiness costs using a fluid approximation which enables closed form derivation of the equilibrium arrival distribution. \citet{JS13} studied the stochastic version of the concert queueing game with a general number of customers, and characterized the symmetric Nash equilibrium arrival distribution in terms of a set of differential equations, and argued that it is absolutely continuous. In particular, they proved that a unique equilibrium exists and it is symmetric. \citet{HR15} studied strategic timing of arrivals to a multi-server loss system, where customers tried to maximize their probability of receiving the service. Comprehensive reviews and summaries of this line of research can be found in \citet{HR20} and \citet[Section 4.1]{H16}.

Most of the literature that deals with queueing process estimation concerns the arrival and service
rates, and describes methods that require continuous observation over an interval of time.  To cite but a few, see for example  \citet{BR87},  \citet{A94}, \citet{AB94}, \citet{BP88}, \citet{BP99}. See \citet{ANT21} for a comprehensive survey of parameter estimation in queues. In \citet{IRM20}, the authors considered a system with strategic customers who can choose to balk after being aware of information about the delay, and estimated the patience distribution and the corresponding potential arrival rate. \citet{RC11} estimated the relative perceived value of patients waiting time, which is a ratio of the cost of the customers’ waiting time to the cost of the server’s idle time. Their method relies on constructing an estimation equation from the optimality condition of the stationary expected cost of the system.  However,  as opposed to our setting, customers are not strategic and the underlying distribution of the queueing process is known.  

\subsection{Outline of the methods and contribution}
\label{subsec:outline}

Before presenting the estimation method we need to first look into the Nash equilibrium dynamics. A symmetric Nash equilibrium is an arrival strategy that it is optimal for any customer if it is also used by all others. We denote the Nash equilibrium distribution by $F_e$, the corresponding probability density by $f_e$. The full details of deriving the symmetric equilibrium can be found in \citet{H13}. We only consider the non-trivial case where $T$ is large enough such that there are customers willing to arrive after time $0$. For the sake of brevity we primarily focus on the \textit{no early bird} variation where there are no arrivals before time $0$ and no service closing time (i.e., $T=\infty$). However, in Section \ref{sec:ext}, we show that with modifications the method can be applied to the other model variations as well. Moreover, the methodology in this work can be extended to other queueing systems with strategic choice of arrival times.

{
Let $C(t), W(t) \geq 0$ be the cost and waiting time, respectively, for a customer who arrives at time $t> 0$,
then
\begin{align*}
C(t)=\alpha W(t)+\beta \, (t+W(t)).
\end{align*}
} If all other customers use $F_e$, her expected cost is 
\begin{align} \label{eq:costconstant}
	\mathbb{E}_{F_e}  [C(t)]=\alpha \, \mathbb{E}_{F_e}  [W(t)]+\beta \, (t+\mathbb{E}_{F_e}[W(t)]) \,.
\end{align}
In equilibrium $\mathbb{E}_{F_e}[C(t)]=c$ for any $t$ in the support of the arrival distribution, where $c>0$ is a constant, and ${\mathbb{E}_{F_e}[C(t)]} \geq c$ for any $t$ not in the support. This type of equilibrium is also known as a Wardrop equilibrium \citep{W52}. 

The first step towards estimating $\theta$ is to estimate the expected waiting time at sampling instants within the support of the arrival distribution, which is a linear function of the expected number of customers in the queue due to the exponential service assumption and the FCFS discipline. Estimation of the support will be explained in detail in Section \ref{subsec:ES}. Suppose $s,t$ are in the support with $0<s <t$, then we can express the expected costs {$ \mathbb{E}_{F_e}  [C(s)]$ and $ \mathbb{E}_{F_e}  [C(t)]$} in terms of $\alpha$, $\beta$ and the expected waiting times of customers arriving at the two times, using Equation \eqref{eq:costconstant}. By setting {$\mathbb{E}_{F_e}  [C(s)]= \mathbb{E}_{F_e}  [C(t)]$} we will show that $\theta$ can be expressed in terms of the expected number of customers at the two time instants:  {As service times are exponential, 
$$\mathbb{E}_{F_e}[W(t)]=\mathbb{E}_{F_e}\left[\mathbb{E}_{F_e}[W(t)|Q(t)]\right]=\frac{\mathbb{E}_{F_e}[Q(t)]}{\mu},$$
where $Q()$ is the queue length process.  Hence, by \eqref{eq:costconstant}, $$\theta=\frac{\mathbb{E}_{F_e}[Q(t)]-\mathbb{E}_{F_e}[Q(s)]}{\mu(s-t)}.$$
This yields a moment estimation equation for any pair of sampling instants $(s,t)$. }If the queue length is sampled at a collection of times, then a refined estimator is given by taking average of the estimators corresponding to all pairs. 

We demonstrate the performance of our estimator through asymptotic analysis and simulation studies.  Specifically, we prove strong consistency and asymptotic normality of the resulting estimator. The proofs rely on establishing the law of large numbers and central limit theorems for the expected queue length estimators and then applying the continuous mapping theorem. We further derive an expression for the asymptotic variance of the estimation error as a function of the covariance matrix of the queue lengths at sampling times. We also explain how the asymptotic variance can be numerically approximated. Remarkably, the asymptotic performance is guaranteed even if the queue length is only sampled three times every day (as long as the sampling times are in the positive support of the equilibrium arrival distribution). However, sampling at multiple instants and spacing these instants far apart can improve the statistical efficiency in terms of the asymptotic variance. This is illustrated through simulation experiments in Section~\ref{section:Estimator}. The main finding of the simulation study is that only a small number of sampling instants are required to reduce the variance of the estimation error. In other words, there is almost no difference between continuous observation of the queue length and sampling it at a small number of time instants.


Our study constructs an estimator of utility parameters of strategic customers by taking advantage of the queueing dynamics brought corresponding to a Nash equilibrium. Moreover, to the best of the authors' knowledge this is the only work to do so for a transient queueing setting. Second, the estimation methodology is applicable for a number of systems with strategic customers whose choice is of the Wardrop equilibria type. 
The setting that customers are strategic is essential in this work, since our method makes use of the properties of expected cost faced by customers arriving according to the Nash equilibrium distribution. Third, the estimator we propose needs substantially less information than standard queueing inference techniques (see \citet{ANT21}).  Namely, it does not require continuous observations, or an accurate estimate of the support of the equilibrium arrival distribution. In fact, with only two points in the support, we are able to estimate the ratio $\theta$. We explain how to choose the two points in detail in Section \ref{section:Estimator}. Moreover, as explained above, simulation analysis suggests that the estimator is robust to the number of sampling instants each day. 


The remainder of the paper is organised as follows. In Section \ref{section:PL} we explain the Nash equilibrium arrival distribution and how to numerically calculate it. In Section \ref{section:Estimator} the estimator for $\theta$ is presented for the case where there are no arrivals before time $0$. Section \ref{section:Aanalysis} presents the asymptotic analysis of the estimator. In Section \ref{section:NumericalExamples} we provide several numerical examples to demonstrate the effect of the number of the observation days and the sampling instants on the estimator. In Section \ref{sec:ext} we modify the estimator for the cases where there are arrivals before time $0$ or a service closing time. Section \ref{sec:ext} further discusses the potential to extend our estimation method to more elaborate systems with strategic arrival time choice.

\section {Preliminaries} \label{section:PL}

In this section, we summarize the results of \citet{H13} and \citet{JS13} that characterize the Nash equilibrium arrival distribution and the associated expected costs. In Appendix~\ref{subsec:da} we illustrate how to approximate the equilibrium distribution and the expected cost for the case where there are no arrivals before time $0$, via the finite difference method. 
 

Let $Q(t)$ be the queue length at time $t$, and $P_k(t) \equiv P(Q(t) = k)$ for $k\geq 0$. In the case where there are no arrivals before time $0$, if there is no closing time, there exist $t_b > t_a > 0$ such that the equilibrium arrival distribution is given by
\begin{align} 
	&f_e(t) = 
	\frac{\mu}{\lambda} \, (1-P_0(t) - \theta) \ \qquad t \in [t_a, t_b] , \label{eq:wof1}\\
	& F_e(0) = p_e \equiv 1- \int_{t = t_a}^{t_b} f_e(t) \, dt \label{eq:ed2}   \,.
\end{align}
That is, there is an atom of size $p_e>0$ at time zero, no arrivals during the interval $(0,t_a)$, a positive density $f_e(t)$ along $[t_a,t_b)$, and $f_e(t) = 0$ for $t \geq t_b$.  

Note that if there is a service closing time $T$, then Nash equilibrium has three cases \citep[Theorem 3.2]{H13}. If $T< t_a$, where $t_a$ is determined by \eqref{eq:wof1} and \eqref{eq:ed2}, then the Nash equilibrium is to arrive at time zero with probability $1$. We do not consider this case in our paper. If $t_a \leq T < t_b$, the equilibrium arrival distribution is similar to the case with no closing time, but the arrival distribution support is $\{0\} \cup [t_a, T]$ and $f_e(T) > 0$. In particular, the values of $p_e$, $f_e(t)$, and $t_a$ satisfy
\begin{align}
	& f_e(t) = \frac{\mu}{\lambda} \, \left(1-P_0(t) - \theta \right), \qquad t \in [t_a, T], \label{eq:wc1} \\
	& p_e = 1- \int_{t = t_a}^{T} f_e(t) \, dt \,.
\end{align}
If $T \geq t_b$, the equilibrium arrival distribution is the same as the case with no service closing time. 

In the case where arrivals before time zero are allowed and there is no closing time, there exist $w, t_w > 0$ such that the equilibrium arrival distribution has
\begin{align}
	&f_e(t) = 
	\begin{cases} \label{eq:wf1}
		\displaystyle\frac{\mu}{\lambda} \, (1-\theta), & -w \leq t <0, \\[+6pt]
		\displaystyle\frac{\mu}{\lambda}  \,  (1-P_0(t) - \theta), & 0 \leq t \leq t_w,\, 
	\end{cases} \\
	& \int_{0}^{t_w} f_e(t)dt = 1-w \, \frac{\mu}{\lambda} (1-\theta)   \label{eq:wf2}\,.
\end{align}
There is a constant density along $[-w,0]$, and positive density along $[0,t_w]$. Note that $P_0(t)$ in \eqref{eq:wof1}, \eqref{eq:wc1}, and \eqref{eq:wf1} depends on the corresponding Nash equilibrium arrival distribution. It can be seen from Equation \eqref{eq:wof1}-\eqref{eq:wf2} that the distribution $F_e$ is determined by the ratio $\theta$, not the specific values of $\alpha$ and $\beta$. 
For all of the above cases \citet{JS13} proved that the equilibrium solution $F_e$ is unique and absolutely continuous (excluding a possible discontinuity at zero).

In the following analysis, we mainly focus on the case where there are no arrivals before time $0$ and no service closing time. Later, in Section~\ref{sec:ext}, we show that the estimator we propose can be easily adapted for other cases. 
Let $\mathbb{E}_{F_e}[C(t)]$ be as defined in Equation \eqref{eq:costconstant}, then a Nash equilibrium $F_e$ satisfies 
\begin{equation}\label{eq: NE1}
	\mathbb{E}_{F_e}[C(t)] =c, \qquad  \forall \, t\in \{0\} \cup [t_a,t_b]\,,
\end{equation}
and
\begin{equation} \label{eq:cost2}
	\mathbb{E}_{F_e}[C(t)]\geq c, \qquad \forall \, {t\notin \{0\}\cup [t_a, t_b]}\,.
\end{equation}
This is because a customer only randomize arrival time points if they result in the same cost. Secondly, the expected cost of a customer arriving at any point outside $\{0\} \cup [t_a,t_b]$ cannot be lower. Otherwise, a customer would be better off by deviating to such an arrival time.

The expected cost for the case without early birds can then be written as
\begin{equation} \label{eq:cost1}
	\mathbb{E}_{{F_e}}[C(t)]=
	\begin{cases}
		\displaystyle\frac{\alpha+\beta}{2\mu} \mathbb{E}_{F_e}[Q(0)],  & t = 0, \\[+6pt]
		\displaystyle\frac{\alpha+\beta}{\mu}\mathbb{E}_{{F_e}}[Q(t)]+\beta t,   & t> 0 \,.
	\end{cases}
\end{equation}
Observe that at time zero there is a positive mass of customers arriving together that are served in random order. Hence, a customer arriving at time zero will on average have to wait for half of the other customers that also arrived with her. To numerically obtain the equilibrium distribution $F_e$ and the associated cost, we employ a finite difference approximation. In particular, we assume that customers can choose to arrive at a time on a discrete grid $\mathcal{T} \equiv \{0, \delta, 2\delta, \ldots\}$ for a small $\delta>0$. Appendix~\ref{subsec:da} provides details of this approximation.

\section{Estimation of $\theta$} 
\label{section:Estimator}

{From now on we assume that customers arrive according to the symmetric Nash equilibrium arrival distribution $F_e$. We denote $q(t)=\mathbb{E}_{F_e}[Q(t)]$, while keeping in mind that all expectations are still with respect to the probability measure imposed by the arrival distribution $F_e$.} Assume that the system manager knows the value of $\mu$, but not $\theta$. It follows from Equation \eqref{eq:wof1} that $F_e$ is determined by $\theta$, $\lambda$, and $\mu$. Note that if $\mu$ is unknown it can be easily estimated, for example by observing a collection of service times. Moreover, the following estimation method does not require knowledge of $\lambda$. Thus, the distribution of the queue length process is a function of $\theta$ via $F_e$. The question is how to construct an estimator for $\theta$ from observations of the queue length.
By the definition of the Nash equilibrium, when others arrive according to $F_e$, each customer is indifferent between arriving at any $t \in \{0\} \cup [t_a,t_b]$. Hence,  Equations \eqref{eq: NE1} and \eqref{eq:cost1} imply that
\begin{equation} \label{eq:theta0}
	\frac{(\alpha+\beta)}{\mu} \,q(t) + \beta \,t \, = \, \frac{(\alpha+\beta)}{2 \, \mu} \, q(0) \, = \, c \qquad \forall \, t \in[t_a, t_b] \,.
\end{equation}
Then we can write down the expression for the expected cost of any two points in $\{0\} \cup [t_a,t_b]$, and derive $\theta$ in terms of the expected queue length at these two points. For any $t \in [t_a,t_b]$, we have
\begin{align} \label{eq:theta01}
	& \frac{(\alpha+\beta)}{\mu} \,q(t) + \beta \,t \, = \, \frac{(\alpha+\beta)}{2 \, \mu} \, q(0), \, \\
	& \Leftrightarrow  \, \, \theta = \frac{\beta}{\alpha+\beta} = -\frac{q(t)  - q(0) /2}{\mu \, t} \,, \notag
\end{align}
and for $s,t \in [t_a,t_b]$,
\begin{align} \label{eq:theta02}
    & \frac{(\alpha+\beta)}{\mu} \,q(s) + \beta \,s \, = \, \frac{(\alpha+\beta)}{\mu} \,q(t) + \beta \,t , \\
    &	\Leftrightarrow  \,\, \theta \, = -\frac{q(s) -q(t)}{\mu \, (s-t)} \,. \notag
\end{align}  
We conclude that $\theta$ can be expressed as a function of the expected queue length at any two observation instants in $\{0\} \cup [t_a,t_b]$. This enables the construction of a method of moments type estimator for $\theta$ from any two observation times. Multiple estimators are obtained when the queue length is observed at more than two time instants, and in that case the mean of all estimators can be taken. In particular, the queue length is inspected at the collection of times $\{t_i\}_{i=1}^m$, where $m>2$. The sampling process is repeated for $n\geq 1$ days, with the same sampling times every day. Towards the construction of the estimator we first detail the necessary preliminary steps: estimation of the expected queue lengths and of the support of the equilibrium arrival distribution, i.e., the values $t_a$ and $t_b$.

\subsection{Estimation of the expected queue lengths}

{Given the equilibrium arrival distribution $F_e$,  we can calculate the expected queue length $q(t)=\mathbb{E}_{F_e}[Q(t)]$. }
Algorithm \ref{alg:ql} in the appendix outlines how this is done using the discrete approximation scheme for the computation of $F_e$. Figure \ref{fig:qs} depicts $q(t)$, two observations can be made. Firstly, $q(t)$ is a linear function on $[t_a , t_b ]$, whose slope is $-\theta \mu$, this can also be derived from Equation \eqref{eq:theta02}. Second, if we extend the line to $t = 0$, then the point intersecting with $t = 0$ is $q(0)/2$, which can also be inferred from Equation \eqref{eq:theta01}. 

\begin{figure}[ht]
	\subcaptionbox{$\alpha = 2, \beta = 0.2, \lambda = 5, \mu= 1$}%
	{\includegraphics[width=0.5\linewidth]{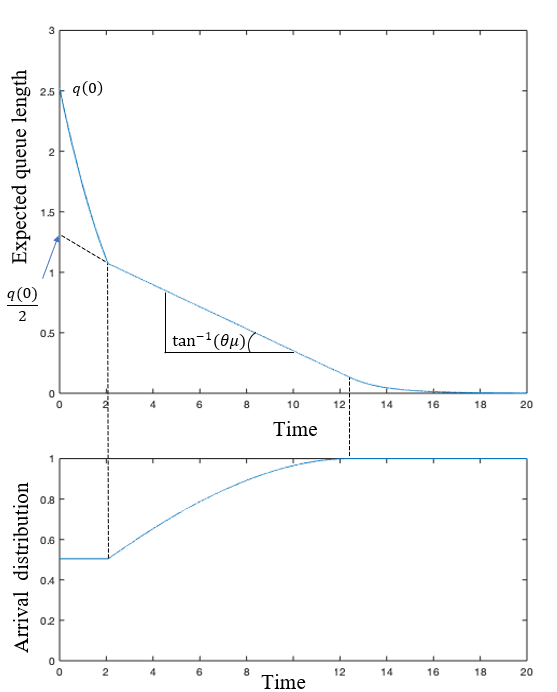}}
	\subcaptionbox{$\alpha = 2, \beta = 0.1, \lambda = 10, \mu= 1$}%
	{\includegraphics[trim={0cm 0.1cm 0cm 0cm},width=0.5\linewidth]{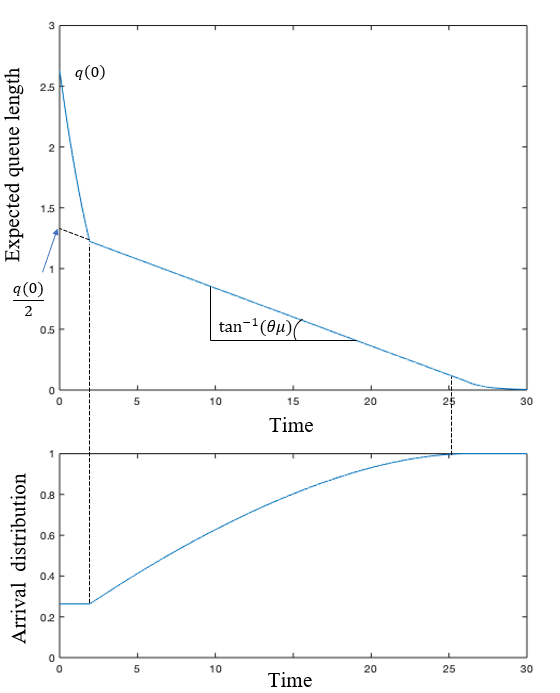}} 
	\caption{The Nash equilibrium arrival distribution and the expected queue length for the model with no arrivals before time zero and no service closing time.} \label{fig:qs}
\end{figure}

Equation \eqref{eq:theta01} and \eqref{eq:theta02} imply that two observation points in the support each day are enough to estimate $\theta$.  However,  observe that to use \eqref{eq:theta01} or \eqref{eq:theta02}, we first need to estimate $[t_a,t_b]$. To achieve this we need more than two observations as will be shown in Section \ref{subsec:ES}. 

Suppose we choose a collection of time instants to observe the queue length every day. Let $n$ and $m > 2$ be the number of observation days and the number of sampling instants every day, respectively. For $l = 1, \ldots, n, \,  i = 1, \ldots, m$, we denote by $t^{(l)}_i$ and $\Xi^{(l)}_i$ the $i$th observation time point and the queue length at that point on the $l$th day.  We assume that the observations are obtained at the same time instants every day, so we drop the index $l$ on $t_i$ and denote the samples on the $l$th day as $\{t_i, \, \Xi^{(l)}_i\}_{i = 1}^{m}$.  {Note that sampling at exactly the same time instants ensures that the observed daily queue length vectors are independent and identically distributed.}
It follows from our assumption of statistical independence between days that $\{t_i, \, \Xi^{(l)}_i\}_{i = 1}^{m}$ are from independent realisations of the process. Since $\Xi_i^{(1)}, \Xi_i^{(2)}, \dots, \Xi_i^{(n)}$ are independent and identically distributed as $Q(t_i)$, the expected queue length can be estimated by the sample mean
\[
\hat{q}_n(t_i) \equiv \displaystyle \frac{1}{n}\sum_{l = 1}^{n} \Xi^{(l)}_i
\] 
at the $i$th sampling time instant. 

Figure \ref{fig:qsn} depicts $\hat{q}_n(t_i)$ for three simulations with $n = 1000$ and $\{t_i\}_{i = 1}^{m} = \mathcal{T}$, compared to the expected queue length $q(t)$. Note that both the simulation and the expected queue length first require numerical evaluation of $F_e$ using Algorithm \ref{alg:wo}. Observe that the queue length estimators yield an unbiased, but noisy, representation of the actual expected queue length process. It is therefore important to use multiple time instants throughout the day in order to reduce the variance of the estimation error. For illustration purposes a very fine grid is used ($m = T_s/\delta+1 = 20001$), however we will later show that a very coarse grid is actually sufficient for achieving very good estimation accuracy. The intuition here is that the queue lengths at very close time instants are highly correlated and thus it is more informative to sample the process with {a bigger space} between observation instants.

\begin{figure}
	\centering
	\includegraphics[width=0.75\linewidth]{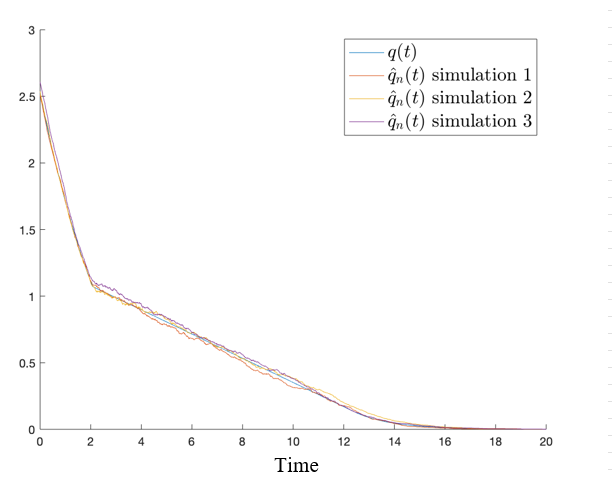}	\caption{The sample mean of the queue lengths ($\alpha = 2, \beta = 0.2, \lambda = 5, \mu= 1$).}  \label{fig:qsn}
\end{figure}

\begin{remark}
We focus on the model with no service closing time, i.e., $T = \infty$. Of course, for practical purposes we must impose a simulation stopping time, denoted by $T_s$. Intuitively, for a good choice there should be a non-empty time interval that does not have any arrival for any of the simulation iterations. Indeed, if $T_s$ is greater than $t_b$, then the strategic behavior of the customers is not affected by the choice of $T_s$. We set $\mathcal{T} = \{0,\delta,2\delta, \ldots, T_s\}$ for all simulation analysis in the remainder of the paper, that is, $\mathcal{T}$ is the grid we run our simulation on. For the parameter configurations $\alpha = 2, \beta = 0.2, \lambda = 5, \mu= 1$ and $\alpha = 2, \beta = 0.1, \lambda = 5, \mu= 1$ we used $T_s = 20$, since we know from the {numerical evaluation of $F_e$ that $t_b < 20$ for these configurations}. In practice, choosing the simulation time is a trial-and-error process. Note that even if $T_s < t_b$, the estimator of $t_b$ is biased, but the estimation of $\theta$ is not (as long as there is a non-empty time interval with arrivals).
\end{remark}

\subsection{Estimation of the support}
\label{subsec:ES}

\begin{figure}[h!]
	\centering
	\subcaptionbox{$t_a = t_{\tilde{a}} <  t_{\hat{a}_n}, \, t_{\hat{b}_n}< t_{\tilde{b}} = t_b$}%
	{\includegraphics[width=0.5\linewidth]{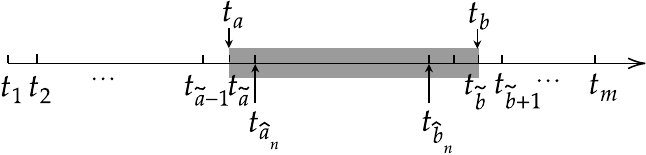}}
	\subcaptionbox{$t_a < t_{\tilde{a}} <  t_{\hat{a}_n},\, t_{\hat{b}_n} < t_{\tilde{b}} < t_b$}%
	{\includegraphics[width=0.5\linewidth]{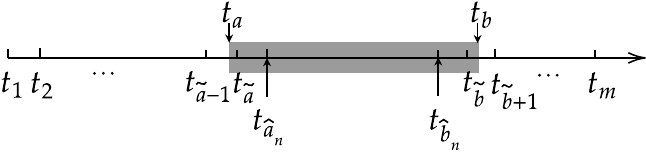}}
	\subcaptionbox{$t_a < t_{\tilde{a}} <  t_{\hat{a}_n},\, t_{\hat{b}_n} = t_{\tilde{b}} < t_b$}%
	{\includegraphics[width=0.5\linewidth]{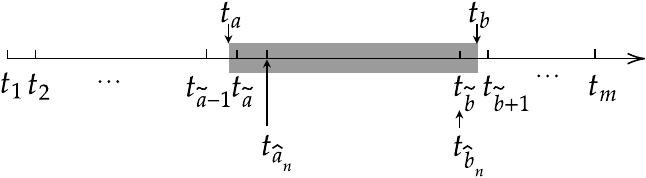}} \\
	\subcaptionbox{$t_a < t_{\tilde{a}} = t_{\hat{a}_n},\, t_{\hat{b}_n} = t_{\tilde{b}} < t_b$}%
	{\includegraphics[width=0.5\linewidth]{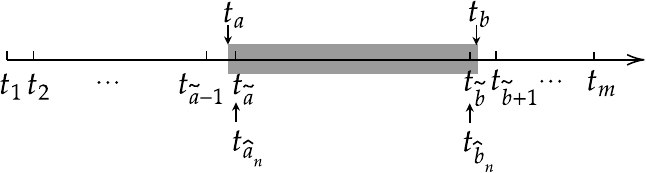}}
	\caption{An illustration of $t_{\hat{a}_n}$, $t_{\hat{b}_n}$, $t_{\tilde{a}}$, and $t_{\tilde{b}}$.} \label{fig:TimeLine}
\end{figure}

To estimate $t_a$ and $t_b$ we must identify the first and last time instants in $\{t_i\}_{i = 1}^m$ such that the queue length increases. By noticing that $[t_a,t_b]$ must include these times, we estimate the support boundaries by
\begin{align}
	\hat{a}_n &= \min_{1 \leq l \leq n} \{\inf_{2 \leq i \leq m} \{i: \Xi^{(l)}_{i} -\Xi^{(l)}_{i-1} \geq 1\} \} \label{eq:boun1},\\
	\hat{b}_n &= \max_{1 \leq l \leq n} \{\sup_{1 \leq i \leq m-1} \{i: \Xi^{(l)}_{i+1} -\Xi^{(l)}_{i} \geq 1\}\}\label{eq:boun2} \,.
\end{align} 
Denote the boundaries of the intersection of the support and the discrete observation grid 
by
\begin{align}
	& \tilde{a} \equiv \inf_{1 \leq i \leq m} \{i: t_i \geq t_a\} \label{eq:tildea},\\
	& \tilde{b} \equiv \sup_{1 \leq i \leq m} \{i: t_i \leq t_b\} \label{eq:tildeb} \, ,
\end{align}
and observe that $[t_{\tilde{a}},t_{\tilde{b}}]\subseteq[t_a,t_b]$. Due to the discrete observation scheme, one cannot accurately estimate $t_a$ and $t_b$. However, we can estimate $t_{\tilde{a}}$ and $t_{\tilde{b}}$ by $t_{\hat{a}_n}$ and $t_{\hat{b}_n}$,  respectively. Figure \ref{fig:TimeLine} describes some examples of $t_{\hat{a}_n}$ and $t_{\hat{b}_n}$ with $t_a, t_b, t_{\tilde{a}},t_{\tilde{b}}$ plotted as well, where the shaded interval is $[t_a, t_b]$. In Figure~\ref{fig:TimeLine} (a), the observation instants $\{t_i\}_{i = 1}^m$ includes both $t_a$ and $t_b$, but this is only possible for carefully chosen parameters such that the boundaries of the discrete grid coincide with those of the support. In this case, $t_{\tilde{a}} = t_a$ and $t_{\tilde{b}} = t_b$. In general $\{t_i\}_{i = 1}^m$ does not include $t_a$ or $t_b$, as in Figure \ref{fig:TimeLine} (b) (c) (d). In any case, it follows from \eqref{eq:boun1} and $\eqref{eq:boun2}$ that
\[
t_a \leq t_{\tilde{a}} \leq  t_{\hat{a}_n}, \qquad t_{\hat{b}_n} \leq t_{\tilde{b}} \leq t_b \,.
\]

Recall that the reason for estimating the support is to be able to select time points inside $[t_a, t_b \,]$ and use the average number of customers at these time points to estimate $\theta$. Even if the estimation for the support is biased, the estimation for $\theta$ is not affected by the bias as long as the chosen pair is inside $[t_a, t_b \,]$, which always holds since $[t_{\hat{a}_n}, t_{\hat{b}_n}] \subset [t_a,t_b]$. 

\subsection{Estimation of $\theta$ using a reference point}
\label{sec:ref_point}

Let $\hat{\mathcal{T}} \equiv \left(\{0\} \cup [t_{\hat{a}_n}, t_{\hat{b}_n}] \right)\cap \{t_i\}_{i=1}^{m}$, where $[t_{\hat{a}_n}, t_{\hat{b}_n}]$ depends on $\{t_i\}_{i=1}^{m}$. The interpretation of $\hat{\mathcal{T}}$ is the observation time slots in the estimated support $\{0\} \cup [t_{\hat{a}_n}, t_{\hat{b}_n}]$. Applying \eqref{eq:theta01} and \eqref{eq:theta02} yields an estimator for $\theta$, 
\begin{equation} \label{eq:est1}
	\hat{\vartheta}_n (t_i,t_j) = 
	\begin{dcases}
		-\frac{\hat{q}_n(t_i) - \hat{q}_n(t_j)}{\mu(t_i-t_j)} &  t_i, t_j > 0, \, t_i\neq t_j, \, t_i, t_j \in \hat{\mathcal{T}} \\
		-\frac{\hat{q}_n(t_i) - \hat{q}_n(0)/2}{\mu t_i}  &   t_i > 0, \, t_j = 0, \, t_i,0 \in \hat{\mathcal{T}}
	\end{dcases}  \,.
\end{equation}
Note that symmetry implies $\hat{\vartheta}_n (t_i,t_j) = \hat{\vartheta}_n (t_j,t_i)$.  {Further observe that for a finite sample there is a positive probability that some of the estimators fall outside of the interval $[0,1]$. In the sequel we verify that the probability of such an event vanishes as the sample size grows. }

To evaluate the performance of this estimator with respect to a choice of a reference point, we consider two parameter configurations and simulated data. The estimators for the boundary indices $\hat{a}_n$ and $\hat{b}_n$ are obtained using Equations \eqref{eq:boun1} and \eqref{eq:boun2}. We then selected a time point $t_r (> 0) \in \hat{\mathcal{T}}$ as a reference point. For each $t \in \hat{\mathcal{T}}/\{t_r\}$, the estimator $\hat{\vartheta}_n(t,t_r)$ is calculated. The two examples of the sequences of estimators $\hat{\vartheta}_n(t,t_r)$ for $t \in \hat{\mathcal{T}}/\{t_r\}$ where $\hat{\mathcal{T}} =\left(\{0\} \cup [t_{\hat{a}_n}, t_{\hat{b}_n}] \right)\cap \mathcal{T}$ and $t_r = 7$, under different parameter settings are depicted in Figure \ref{fig:thetar}.  In Figure \ref{fig:thetar} (a), $t_a = 2.075, t_b = 12.415$, the estimates $t_{\hat{a}_n}$ and $t_{\hat{b}_n}$ are $2.076$ and $12.210$, respectively. In Figure \ref{fig:thetar} (b), $t_a = 1.523, t_b = 16.594$, the estimates $t_{\hat{a}_n}$ and $t_{\hat{b}_n}$ are $1.538$ and $16.324$, respectively. The true $\theta$ value is plotted by the dotted line, and the $*$ at $t = 0$ is $\hat{\vartheta}_n(0,t_r) (=\hat{\vartheta}_n(t_r,0))$. It can be observed from both plots that the closer the two points are, the more biased the estimation is. The explanation is that the dependence between the queue lengths at two time slots increases when the distance between them decreases. In Section \ref{section:Aanalysis}, we will see the estimator variance increases in quadratic rate with the distance. 

\begin{figure}[H]
	\centering
	\subcaptionbox{$\alpha = 2, \beta = 0.2, \lambda = 5, \mu= 1, t_r =7.$}{\includegraphics[width=0.5\linewidth]{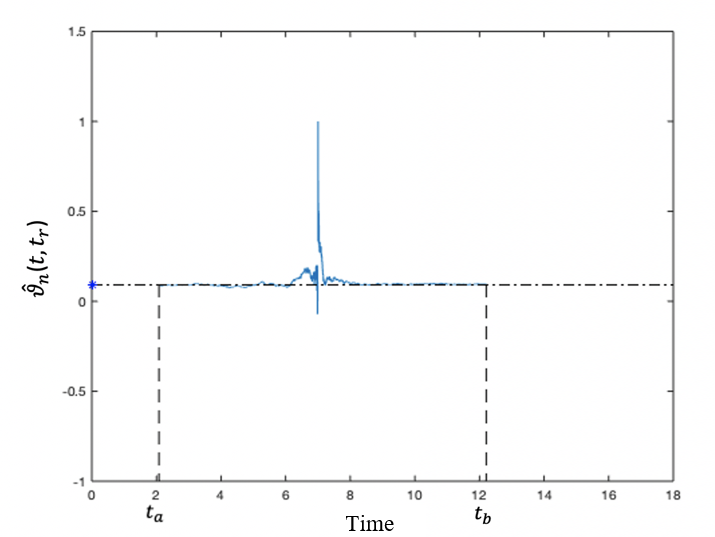}}%
	\subcaptionbox{$\alpha = 2, \beta = 0.1, \lambda = 5, \mu= 1, t_r = 7.$}{\includegraphics[width=0.49\linewidth]{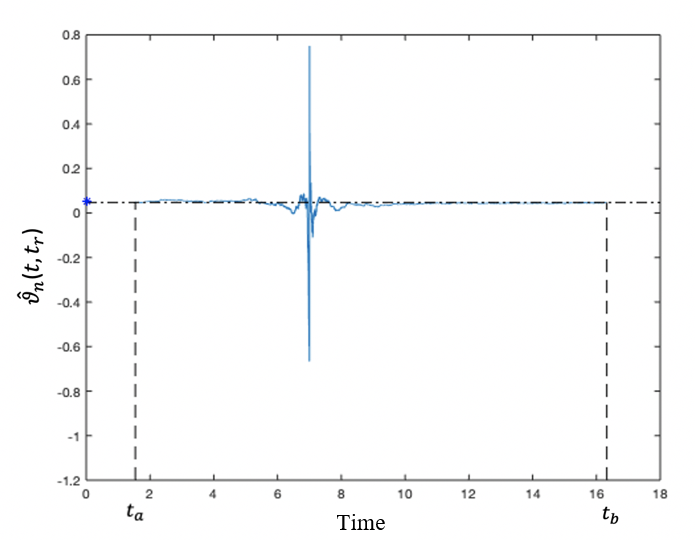}}
	\caption{Estimates of $\theta$ using a reference instant.}  \label{fig:thetar}
\end{figure}

\begin{remark} It follows from Equation \eqref{eq:wof1} and $f(t_b) = 0$ that $\theta$ can also be expressed as
\begin{equation} \label{eq:est2}
	\theta = 1-P_0(t_b) \,.
\end{equation}
This means if we can estimate $t_b$ by  $t_{\hat{b}_n}$, we can estimate $P_0(t_b)$ by counting how many empty queues there are at $t_{\hat{b}_n}$ and dividing it by the total sample size $n$, thus yielding an estimator for $\theta$. However, this requires an accurate estimate of $t_b$ first, which requires both a large $m$ such that $\{t_i\}_{i=1}^m$ includes a point that is very close to it, and a large sample size $n$. Moreover, for any $m<\infty$ the estimator for $t_b$ is biased so it may not be possible to obtain a consistent estimator. Also, in the model with service closing time, which we discuss in Section \ref{sec:ext}, there does not necessarily exist a time $t$ that satisfies $f(t) = 0$, thus there is no expression equivalent to \eqref{eq:est2}.
\end{remark}


\subsection{The mean estimator}
\label{subsec:E}

\begin{figure}[h!]
	\centering
	\subcaptionbox{$\alpha = 2, \beta = 0.2, \lambda = 5, \mu= 1.$}{\includegraphics[width=0.495\linewidth]{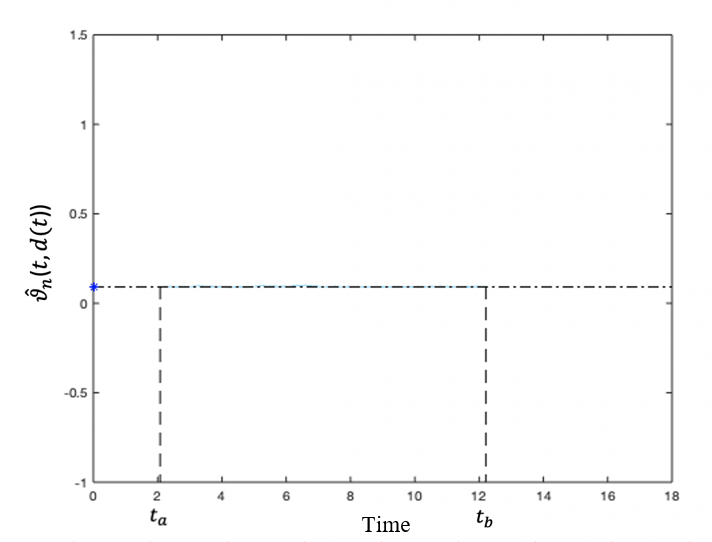}} 
	\subcaptionbox{$\alpha = 2, \beta = 0.1, \lambda = 5, \mu= 1.$}{\includegraphics[width=0.495\linewidth]{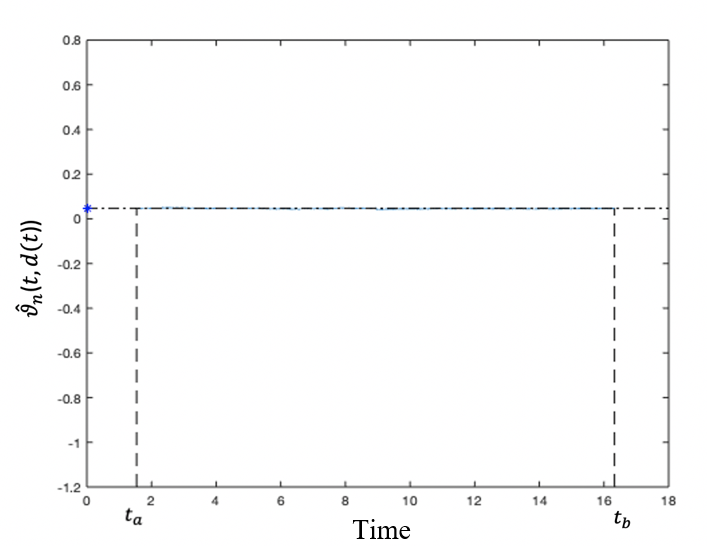}} 
	\caption{Estimates of $\theta$ using the farthest instant.} \label{fig:thetad}
\end{figure}

With the aforementioned reasoning in mind, it is better to choose two observation points that are far away from each other. For any $t \in \hat{\mathcal{T}}$, let $d(t) = \arg\max_{t_i \in \hat{\mathcal{T}}} |t_i-t|$ and we compute $\hat{\vartheta}_n(t, d(t))$ for $t \in \hat{\mathcal{T}}$. The results for the examples considered in the prequel are depicted in Figure \ref{fig:thetad}. The true $\theta$ value was plotted by the dotted line, and the $*$ at $t = 0$ is $\hat{\vartheta}_n(0,d(0))$. It is observed from the plots that $\hat{\vartheta}_n(t, d(t))$ is generally better than $\hat{\vartheta}_n(t, t_r)$ for a fixed reference point $t_r$. This is specially noticeable for $t$ near $t_r$. Denote the cardinality of $\hat{\mathcal{T}}$ by $|\hat{\mathcal{T}}|$. When $|\hat{\mathcal{T}}| > 2$, we propose an estimator
\begin{equation}\label{eq:theta2}
	\hat{\theta}_n \equiv \frac{1}{| \hat{\mathcal{T}} |} \, \sum_{t \in \hat{\mathcal{T}}} \hat{\vartheta}_n(t, d(t)) \,. 
\end{equation}
by taking the mean of the estimates calculated by every pair in $\hat{\mathcal{T}}$.

\begin{figure}
	\centering
	\includegraphics[width=0.9\linewidth]{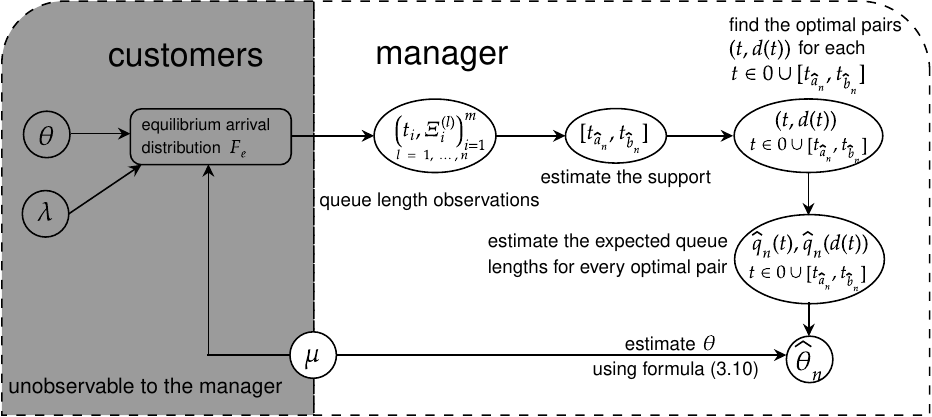}
	\caption{An overview of the observation and estimation process.} \label{fig:estimation}
\end{figure}

An overview of the observation and the estimation process is depicted in Figure \ref{fig:estimation}. The manager uses her observations and the known value of $\mu$ to estimate $\theta$. First the expected queue lengths are estimated from independent realizations of the queue. Next the interval of continuous arrivals $[t_a,t_b]$ is estimated by considering the time points closest to the boundaries on the discrete observation grid $\mathcal{T}$. Note that although the observations from different days are independent, the observations at different times for a single day are not, hence there is a need to carefully select the pairing of observation instants in the construction of the sequence of estimators $\hat{\vartheta}_n (t_i,t_j)$ for $\theta$. To this end a heuristic rule of choosing the farthest away observation (in terms of time) is applied to every sampling time instant. The intuitive explanation for this rule is that the correlation between queue lengths at different sampling times decreases the further apart they are chosen. This intuition is verified by the examples presented above and simulation experiments in Section~\ref{section:NumericalExamples}. Finally, the estimator $\theta$ in \eqref{eq:theta2} is given by taking the mean of the sequence of estimators obtained by the optimal pairing rule.

\section{Asymptotic  analysis} \label{section:Aanalysis}
We first state the main results of this section with the proofs detailed in the following subsections. In the following, $\mathcal{N} \left(\bm{\mu}, \Sigma\right)$ denotes a normally distributed random variable with mean vector $\bm{\mu}$ and covariance matrix $\Sigma$. An estimator is said to be strongly consistent if as the number of observation days increases, the resulting sequence of estimates converges almost surely to the true value.  Theorem \ref{thm:consistency} proves the strong consistency of our estimator $\hat{\theta}_n$. We establish the asymptotic normality of our estimator $\hat{\theta}_n$, and prove in Theorem \ref{theorem:ad} that as $n \rightarrow \infty$, the estimation error scaled by $\sqrt{n}$ converges to a zero-mean normal random variable, whose variance can be numerically approximated. We explain in detail how the variance is calculated in Section \ref{subsec:VarianceC}. The proofs of Theorem \ref{thm:consistency} and \ref{theorem:ad} are provided in Section \ref{subsec:StrongConsistency} and Section \ref{subse:AsymptoticD}, respectively. 

We assume the sampling points do not include time $0$ in this section. The analysis is similar if $0$ is included.
 
\begin{theorem}
	\label{thm:consistency}
	As $n\to\infty$, $\hat{\theta}_n  \rightarrow_{a.s.}\theta$.
\end{theorem}
\begin{theorem} 
	\label{theorem:ad}
	For $t_i \in \hat{\mathcal{T}}$, let $k_i \equiv |\hat{\mathcal{T}}| \left(t_i - d(t_i) \right) \mu$, $g_i \equiv \sum_{j \neq i}\frac{1}{k_j} \mathbbm{1}_{\{d(t_j) = t_i\}}-\frac{1}{k_i}$. Let $v(t) \equiv \mathrm{Var}_{F_e}[Q(t)]$ and $\rho(s,t) \equiv \mathrm{Cov}_{F_e}[Q(s),Q(t)]$ be the variance and covariance, for any $s,t\in[t_a,t_b]$. As $n \rightarrow \infty$,
	\[
	\sqrt{n} \left( \hat{\theta}_n - \theta \right) \,  \xrightarrow{d} \mathcal{N}\left( 0 \, , \,  \sum_{i=\tilde{a}}^{\tilde{b}} g_i^2v(t_i)+2\sum_{i=\tilde{a}}^{\tilde{b}} \sum_{j> i}^{\tilde{b}} g_i g_j\rho(t_i,t_j) \right) \,.
	\]
\end{theorem}

\subsection{Strong consistency} \label{subsec:StrongConsistency}

The proof of Theorem \ref{thm:consistency} requires finite first and second moments of the queue length $Q(t)$. The total number of customers in the system is bounded by a Poisson distributed random variables with mean $\lambda$, so as long as $0< \lambda<\infty$, $v(t)<\infty$ and $|\rho(s,t)|<\infty$. 


{\it Proof of Theorem \ref{thm:consistency}.}
The first and second moments of $Q(t)$ are bounded, it follows from Kolmogorov strong law of large numbers \citep[p251]{L77} that
\begin{equation*}
	\hat{q}_n(t) \rightarrow_{a.s.} q(t) \,.
\end{equation*}
Also, it follows from Equation \eqref{eq:est1} and \eqref{eq:theta2} that both $\hat{\vartheta}_n(t,d(t))$ and $\hat{\theta}_n$ are linear functions of $\hat{q}_n(t)$, thus the continuous mapping theorem in \citet[Theorem 2.3]{VV00} implies
\begin{equation*}
	\hat{\vartheta}(t, d(t)) \rightarrow_{a.s.} -\frac{ q(t)- q\left(d(t)\right)}{(t-d(t))\mu} = \theta \ , \qquad \hat{\theta}_n \rightarrow_{a.s.} \frac{1}{m} \sum_{i = 1}^{m} \theta = \theta \,.
\end{equation*}
$\hfill \square$

We have explained in Section \ref{subsec:ES} that the reason to estimate the support $[t_a,t_b]$ is to be able to choose at least two points inside it for the estimation, and our method assures that $[t_{\hat{a}_n}, t_{\hat{b}_n}] \subseteq [t_a,t_b]$. Moreover, with a discrete observation scheme the best one can hope for is to estimate the points on the grid that are closest to the actual boundaries of the support, i.e., $t_{\tilde{a}}$ and $t_{\tilde{b}}$ defined in Equations \eqref{eq:tildea} and \eqref{eq:tildeb}, respectively. Although it is not essential to have accurate estimates of $t_a$ and $t_b$, we prove in the following proposition that $t_{\hat{a}_n}$ and $t_{\hat{b}_n}$ converges to $t_{\tilde{a}}$ and $t_{\tilde{b}}$, respectively.  Proposition~\ref{prop:support} is used in establishing the asymptotic distribution of the errors in the proof of Theorem~\ref{theorem:ad}.
\begin{proposition}\label{prop:support}
	If there are at least two observation points from $\{t_i\}_{i=1}^{m}$ that are inside the support of the arrival distribution, then as $n\to\infty$, $t_{\hat{a}_n}  \rightarrow_{a.s.} t_{\tilde{a}}, \, t_{\hat{b}_n}  \rightarrow_{a.s.} t_{\tilde{b}}$\,. 
\end{proposition}
\begin{proof}
	We label the observation times $\{t_i\}_{i=1}^{m}$ in a way such that $t_{i+1} > t_i$ for $i = 1, \ldots, m-1$. For $t \geq t_a$, $F_e(t)$ is increasing, so $F_e(t_{\tilde{a}+1}) > F_e(t_{\tilde{a}})$. By definition, $t_{\hat{a}_n}$ is greater than or equal to $t_{\tilde{a}}$. We will show that the probability that there are no days with an increase in queue length between the observation instants $t_{\tilde{a}}$ and $ t_{\tilde{a}+1}$ goes to zero. The queue length increases if there are more arrivals than departures in this period. If such an increase is observed on day $n_1$, then $t_{\hat{a}_n}=t_{\tilde{a}}$ for all $n\geq n_1$.
	Given that there are arrivals after time $0$ in equilibrium, the probability that there is at least one arrival but no departures during $[t_{\tilde{a}}, t_{\tilde{a}+1}]$ is at least
	\begin{equation*}
	    \frac{F_e(t_{\tilde{a}+1}) - F_e(t_{\tilde{a}})}{1-p_e} \, e^{-\mu (t_{\tilde{a}+1} - t_{\tilde{a}})} \,.
	\end{equation*}
	 As the observations of different days are independent, we have that 
	\begin{equation*}
		\mathbb{P} \left(t_{\hat{a}_n} > t_{\tilde{a}} \right) \leq \left(1 - \frac{F_e(t_{\tilde{a}+1}) - F_e(t_{\tilde{a}})}{1-p_e} \, e^{-\mu (t_{\tilde{a}+1} - t_{\tilde{a}})}\right)^n \,.
	\end{equation*}
	Therefore,
	\begin{equation*}
		\lim\limits_{n \rightarrow \infty} \mathbb{P} \left( | t_{\hat{a}_n} - t_{\tilde{a}} |> 0 \right) = \lim\limits_{n \rightarrow \infty}  \mathbb{P} \left( t_{\hat{a}_n} - t_{\tilde{a}} > 0 \right) =0 \,.
	\end{equation*}
	Similarly, we have 
	\begin{equation*}
		\mathbb{P}  \left(t_{\hat{b}_n} < t_{\tilde{b}} \right)  \leq \left(1-\frac{F_e(t_{\tilde{b}}) - F_e(t_{\tilde{b}-1})}{1-p_e} \, e^{-\mu (t_{\tilde{b}} - t_{\tilde{b}-1})} \right)^n \,,
	\end{equation*}
	and
	\begin{equation*}
		\lim\limits_{n \rightarrow \infty} \mathbb{P} \left( |t_{\hat{b}_n} - t_{\tilde{b}} |> 0 \right) = \lim\limits_{n \rightarrow \infty}  \mathbb{P} \left(t_{\tilde{b}} - t_{\hat{b}_n} > 0 \right) =0 \,.
	\end{equation*}
	Moreover, convergence in probability of a monotone sequence implies the convergence with probability 1 \citep[Lemma 3.2]{K97}, so we conclude
	\begin{equation*}
		t_{\hat{a}_n} \rightarrow_{a.s.} t_{\tilde{a}} \qquad t_{\hat{b}_n} \rightarrow_{a.s.} t_{\tilde{b}} \,.
	\end{equation*}
\end{proof}

\subsection{Asymptotic distribution of the estimation error} 
\label{subse:AsymptoticD}

{\it Proof of Theorem \ref{theorem:ad}.}
By Proposition~\ref{prop:support} we know that $\{t_i\}_{i=\hat{a}_n}^{\hat{b}_n}$ converges almost surely to a fixed collection of times in $[t_a,t_b]$. That is, the vector 
\[
\hat{\bm{q}}_n= \left[ \hat{q}_n(t_{\hat{a}_n}), \hat{q}_n(t_{\hat{a}_n+1}), \ldots, \hat{q}_n(t_{\hat{b}_n}) \right]
\] 
converges almost surely to the same limit as 
\[
\left[ \hat{q}_n(t_{\tilde{a}}), \hat{q}_n(t_{\tilde{a}+1}), \ldots, \hat{q}_n(t_{\tilde{b}}) \right] \,.
\]
Moreover, $\hat{q}_n(t)$ for any $t \in [t_{\tilde{a}},t_{\tilde{b}}]$ satisfies the Central Limit Theorem  because it is an average of independent and identically distributed observations with a known covariance matrix. Hence,  letting

\[
{\bm{q}} = \left[q(t_{\tilde{a}}), q(t_{\tilde{a}+1}), \ldots, q(t_{\tilde{b}}) \right] \,,
\] 
we have
\begin{align*}
	\sqrt{n}(\hat{\bm{q}}_n-\bm{q})\xrightarrow{d}\mathcal{N}(0,\Sigma)\,,
\end{align*}
where the covariance matrix $\Sigma\in\mathbb{R}^{(\tilde{b}-\tilde{a}+1) \times (\tilde{b}-\tilde{a}+1)}$ is given by
\begin{align*}
    \Sigma_{ii} &= v(t_{\tilde{a}+i-1}) , \\
    \Sigma_{ij} &= \rho \left(t_{\tilde{a}+i-1}, t_{\tilde{a}+j-1} \right) , \ 1 \leq i,j \leq \tilde{b}-\tilde{a}+1. 
\end{align*}

Next, by \eqref{eq:theta2} the estimator can be written as the linear combination 
\begin{align*}
	\hat{\theta}_n =-\sum_{i=\hat{a}_n}^{\hat{b}_n}\frac{\hat{q}_n(t_i)-\hat{q}_n(d(t_i))}{k_i}=\sum_{i=\hat{a}_n}^{\hat{b}_n} g_i\hat{q}_n(t_i)\,,
\end{align*}
where $k_i$ and $g_i$ are as defined in Theorem \ref{theorem:ad}. Again, Proposition~\ref{prop:support} implies $\{t_i\}_{i=\hat{a}_n}^{\hat{b}_n}$ converges almost surely to $\{t_{\tilde{a}}, t_{\tilde{a}+1}, \ldots, t_{\tilde{b}}\}$. Let $\bm{g} = [g_{\tilde{a}},g_{\tilde{a}+1},\ldots,g_{\tilde{b}}]$, the delta method \citep[Chapter 3]{VV00} can be applied to conclude that
\[
\sqrt{n} \left( \hat{\theta}_n - \theta \right) \, \xrightarrow{d} \, \mathcal{N} \left( 0,  \bm{g} \, \Sigma \, \bm{g}^\top \right) \,.
\]

\subsection{The variance computation}
\label{subsec:VarianceC}

As $n \rightarrow \infty$, $\sqrt{n} \left( \hat{\theta}_n - \theta \right)$ converges to a zero-mean normal random variable with variance 
\begin{equation} \label{eq:var}
	\sum_{i=\tilde{a}}^{\tilde{b}} g_i^2v(t_i)+2\sum_{i=\tilde{a}}^{\tilde{b}} \sum_{j> i}^{\tilde{b}} g_i g_j\rho(t_i,t_j) \,.
\end{equation}
The variance can be approximated numerically using a discrete approximation, similar to the one used for computing the expected queue length process. Observe that the unknown components in \eqref{eq:var} are $v(t_i)$ and $\rho(t_i,t_j)$. Following the scheme of Section~\ref{subsec:da}, given $F_e$, we can calculate the queue length dynamics $P_k(t), \, 0 \leq k \leq  K$ for $t \in \mathcal{T}$. With $P_k(t)$, we can compute both the first moment and the second moment of the queue length at $t$ by $\sum_{k=0}^{K} k \, P_k(t)$ and $\sum_{k=0}^{K} k^2 \, P_k(t)$, respectively. Then $v(t_i)$ can be obtained by calculating $\sum_{k=0}^{K} k^2 \, P_k(r\delta)- (\sum_{k=0}^{K} k \, P_k(r\delta))^2$. 

The calculation of $\rho(t_i,t_j)$ is a bit more involved. From its definition, 
\begin{align}
	\rho(t_i,t_j) = \mathbb{E}\left[ Q(t_i) \, Q(t_j) \right] - q(t_i) \, q(t_j)
\end{align}
where the first term can be written as 
\begin{align}
	\mathbb{E}\left[ Q(t_i) \, Q(t_j) \right] &= \sum_{1 \leq k, \, l \leq \infty } k \, l \, \mathbb{P} \left[Q(t_j) = l,  \, Q(t_i) = k\right] \notag \\
	&= \sum_{1 \leq k, \, l \leq \infty} k \, l \, \mathbb{P} \left[Q(t_j) = l  \mid Q(t_i) = k\right] \, \mathbb{P} \left[ Q(t_i) = k \right] \,.
\end{align}
The reason to write $\mathbb{E}\left[ Q(t_i) \, Q(t_i) \right]$ by its conditional probability is that the joint probability cannot be approximated directly. Note that for numerical purposes truncation of the sums is required. The conditional probability $\mathbb{P} \left[Q(t_j) = l  \mid Q(t_i) = k\right]$ can be obtained by calculating $P_l(t_j)$, given $P_k(t_i) = 1$ and $F_e$, which indicates that we need to calculate $P_k(t_i+r \delta), \, 1 \leq k \leq K$ for $r = 1,2,\ldots$ until $r\delta \geq t_j$.

\section{Simulation analysis} \label{section:NumericalExamples}

This section presents simulation analysis of the performance of the estimator. {Let $\kappa = 1000$ be the number of simulations, and denote by $\hat{\theta}_n^{(k)}$ the estimate of the $k$th simulation, for $k=1,\ldots, \kappa$}. We compare the estimates obtained with the same number of observation points $m$ but different sample sizes $n$. We also compare the estimates with the same sample size but different number of observation points each day. We use Box-plots to represent estimates of the 1000 simulations. On each box, the central mark indicates the median, and the bottom and top edges of the box indicate the 25th and 75th percentiles, respectively. The whiskers extend to the most extreme data points not considered outliers, and the outliers are plotted individually using the '+' marker symbol.  In our performance analysis, we compute the average of the estimates (AE) and the {root-mean-square deviation (RMSD)} between $\theta$ and its estimate, which are defined as 
\[
\textbf{AE} =\frac{1}{\kappa} \sum_{k = 1}^{\kappa} \,\hat{\theta}_n^{(k)}, \qquad  \textbf{RMSD} = \sqrt{\frac{1}{\kappa} \sum_{k = 1}^{\kappa} \, \left(\hat{\theta}_n^{(k)} - \theta \right)^2} \,.
\] 
The main result of this section is that the estimator is robust to the number of observation points $m$.


As in the previous section, the simulations rely on the approximation of $F_e$ using the discrete scheme described in Appendix~\ref{subsec:da}. {Let $r_i$ denote the $r_i$th slot on the discrete grid, so $r_i \delta$ is the time of the $i$th observation. E.g. $m = T_s/\delta+1$ means that $r_{i+1} - r_i = 1$ for $i = 1, \ldots, m-1$.} The sample is $\{r_i \delta  , \, \Xi^{(l)}_i\}_{i = 1}^{m}$ with $r_1 = 0$, that is, we always consider time $0$ as an observation point. We first run $1000$ simulations with $\alpha = 2, \beta = 0.2$,  yielding $\theta\approx0.091$, and $\lambda = 5, \mu = 1$, $m = T_s/\delta +1 = 20001$ for $n= 1000, 500, 100, 50$.
{The Box-plots of $\{\hat{\theta}_n^{(k)}\}_{k=1}^{1000}$ are presented in Figure \ref{fig:Boxn}, and the AE and RMSD of the estimates are presented in Table \ref{tab:n} for each $n$.} It can be observed that although the estimator quality in terms of the mean and variance decreases with $n$, the mean does not differ too much, and it is close to the true $\theta (\approx 0.091)$. 

\begin{figure}[htb!] 
\centering 
	{\includegraphics[width=0.6\linewidth]{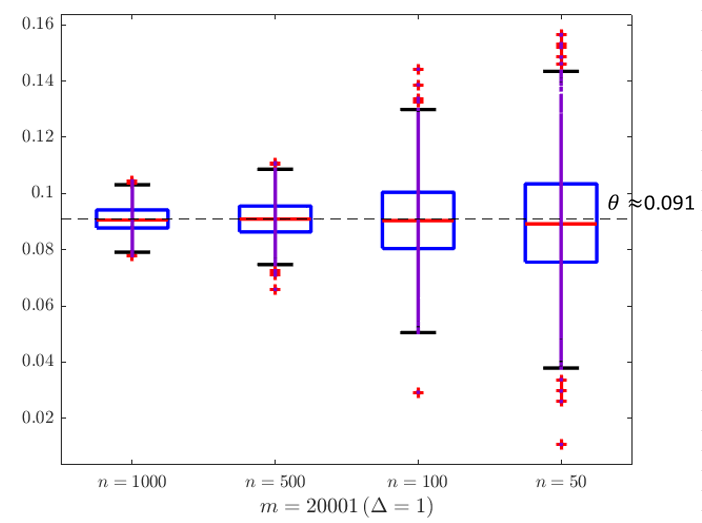}} 
	\caption{Box-plots of $\hat{\theta}$ from 1000 simulations with the same observation instants but different number of days. The true $\theta  \approx 0.091$ is represented by the dotted line.}\label{fig:Boxn}
\end{figure} 
\begin{table}[htb] 
	\scriptsize
	\caption{Characteristics for the estimates in Figure \ref{fig:Boxn}}
	\centering
	\newcolumntype{C}{>{\centering\arraybackslash}X}
	\begin{tabular}{lS[table-format=2.2]*{3}{S}}
		\toprule \toprule
		\addlinespace
		& \multicolumn{4}{c}{\makecell{$m = 20001 \, \, (\Delta  = 1)$}} \\
		\cmidrule(lr){2-5}
		{}&{$n=1000$}&{$n=500$}&{$n=100$}&{$n=50$}
		\tabularnewline
		\cmidrule[\lightrulewidth](lr){1-5}\addlinespace[1ex]
		AE&0.0909 & 0.0909 & 0.0906 & 0.0896 \tabularnewline
		RMSD & 0.0046 & 0.0065 & 0.0152 & 0.0214 \tabularnewline
		\addlinespace
		\bottomrule
	\end{tabular}\label{tab:n}
\end{table}%

We next investigate how the value of $m$ affects the estimates. We assume that the inter observation times, denoted by $\Delta$, are equidistant. That is, $r_1 = 0$, and $\Delta \equiv (r_{i+1} - r_i) \, \delta$ for any $i = 1, \ldots, m-1$. Since we set the simulation time $T_s = 20$, $m = T_s/\Delta+1$. Note that when the value of $m$ or $n$ is small, it is possible that $t_{\hat{a}} > t_{\hat{b}}$
and $\theta$ cannot be estimated. Thus, we use $\eta$ to denote the number of simulations for which $\theta$ was successfully estimated in the $1000$ simulations we tried. We present the results of 1000 simulations with $\alpha = 2, \beta = 0.2, \lambda = 5, \mu = 1, n= 1000$ for $m = 001, 41, 21, 5 \, (\Delta = 0.001,0.5,1,5)$, and show the estimates $\{\hat{\theta}_n^{(k)}\}_{k=1}^{1000}$ in Figure {\ref{fig:Boxm} (a). Figure \ref{fig:Boxm} (b) depicts the asymptotic distribution of $\sqrt{n} \left(\hat{\theta}_n - \theta\right)$ derived in Theorem~\ref{theorem:ad}.} It can be seen that the estimator is robust to the value of $m$. The mean is very close to the true $\theta$ for all values of $m$.  {The variance is larger when the number of observations points $m$ is small.  However, we further observe that increasing $m$ from $21$ does not noticeably impact the variance of the estimation error. The characteristics of the simulations results are summarized in Table \ref{tab:delta} and it can be observed that the RMSD is less than $0.01$} for $m = 20001, 41, 21, 5$, while it is slightly higher for $m=5$. {It is important to point out that Figure \ref{fig:Boxm} (b) is not the output of simulations, but rather a direct computation of the asymptotic variance $\bm{g} \Sigma \bm{g}^\top$ from Theorem~\ref{theorem:ad}. Thus, this step can be carried out when designing a sampling scheme in order to determine the number of sampling instants $m$ required for good accuracy in terms of asymptotic variance of the estimation errors. }

\begin{figure}[ht] 
\centering
	\subcaptionbox{}%
	{\includegraphics[width=0.495\linewidth]{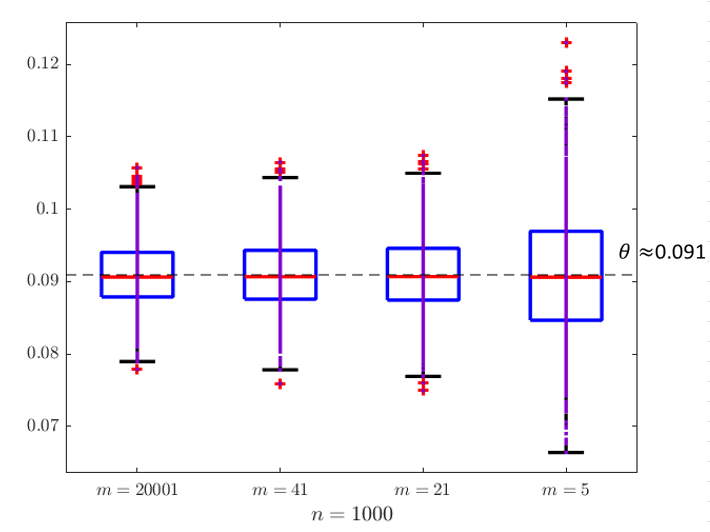}} 
	\subcaptionbox{}%
	{\includegraphics[width=0.495\linewidth]{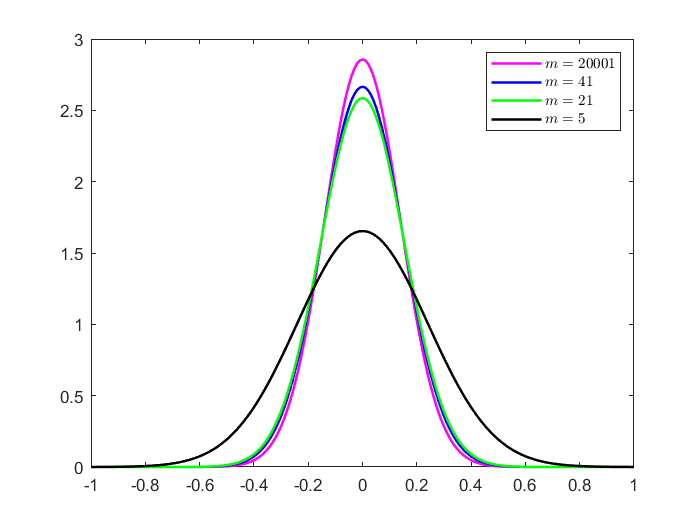}} 
	\caption{(a) depicts Box-plots of $\hat{\theta}_n$ from 1000 simulations with the same observation instants but different number of days. The true $\theta  \approx 0.091$ is represented by the dotted line. (b) depicts the scaled asymptotic distribution $\mathcal{N} (0, \bm{g} \, \Sigma \, \bm{g}^\top)$ in Theorem \ref{theorem:ad} for $m = 20001, 41, 21, 5$.} \label{fig:Boxm}
\end{figure}

\begin{table}[ht] 
	\scriptsize
	\caption{Characteristics for the estimates in Figure \ref{fig:Boxm}.}
	\centering
	\newcolumntype{C}{>{\centering\arraybackslash}X}
	\begin{tabular}{lS[table-format=2.2]*{3}{S}}
		\toprule \toprule
		\addlinespace
		& \multicolumn{4}{c}{\makecell{$n = 1000$}} \\
		\cmidrule(lr){2-5}
		{}&{$m = 20001$ ($\Delta = 0.001$)}&{$m = 41$ ($\Delta = 0.5$)}&{$m = 21$ ($\Delta = 1$)} &{$m = 5$ ($\Delta = 5$)}
		\tabularnewline
		\cmidrule[\lightrulewidth](lr){1-5}\addlinespace[1ex]
		AE&00.0909 & 0.0910 & 0.0910 & 0.0909 \tabularnewline
		RMSD & 0.0046 & 0.0049 & 0.0053 & 0.0093\tabularnewline
		\addlinespace
		\bottomrule
	\end{tabular}\label{tab:delta}
\end{table}%

\begin{figure}[h!] 
	\subcaptionbox{$m = 20001$}%
	{\includegraphics[width=0.508\linewidth]{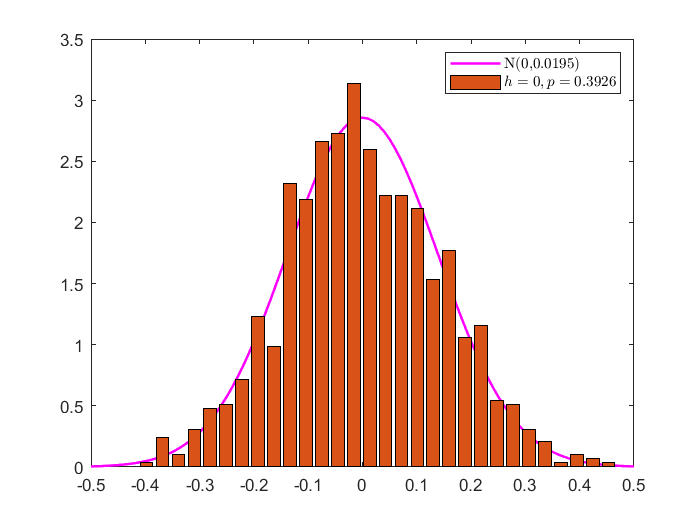}}
	\subcaptionbox{$m = 41$}%
	{\includegraphics[width=0.508\linewidth]{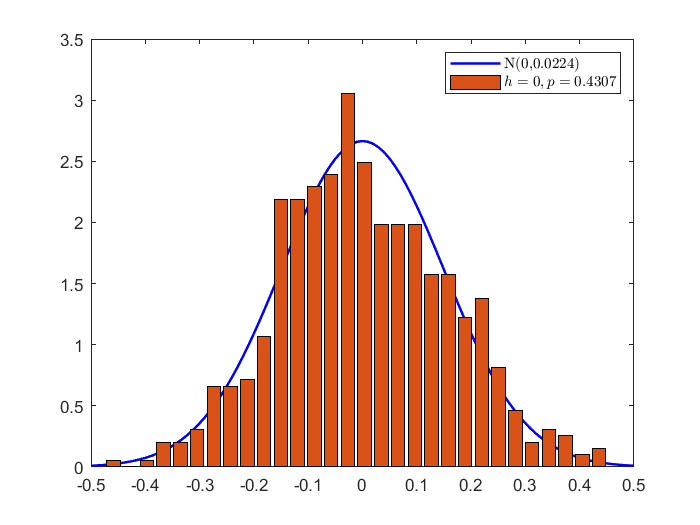}}
	\subcaptionbox{$m = 21$}%
	{\includegraphics[width=0.508\linewidth]{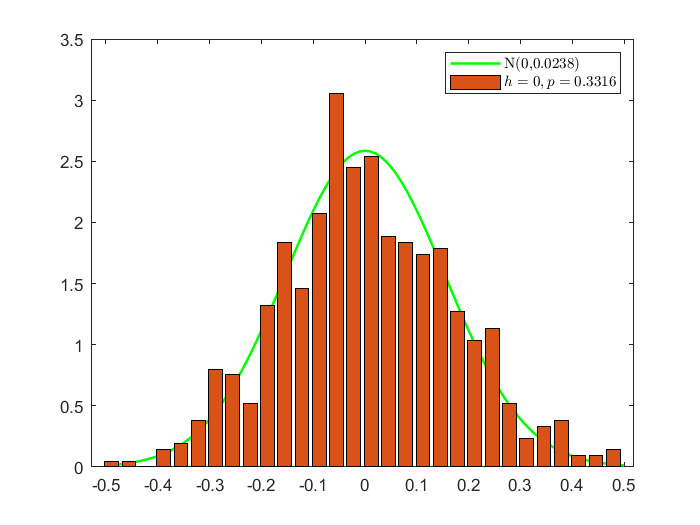}}
	\subcaptionbox{$m = 5$}%
	{\includegraphics[width=0.495\linewidth]{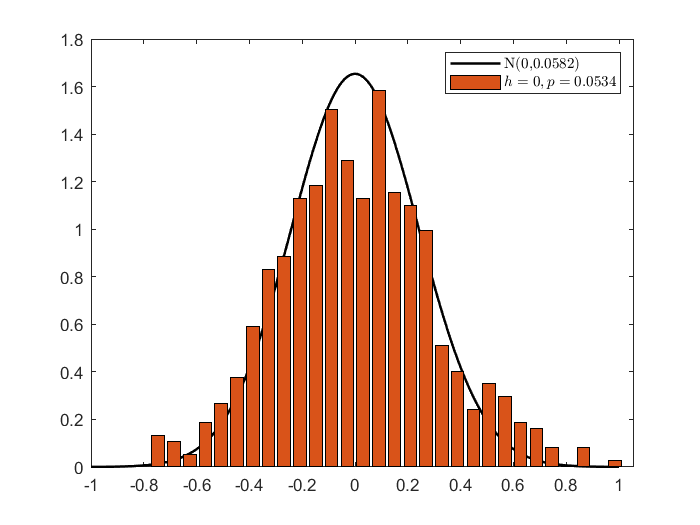}} 
	\hspace{\fill}
	\caption{The normalized histograms of 1000 simulations for $n = 1000$ and different values of $m$. The scaled asymptotic distribution N$(0, \bm{g} \, \Sigma \, \bm{g}^\top)$ in Theorem \ref{theorem:ad} is also plotted. We also did normality test, and present the } \label{fig:hist}
\end{figure}

\begin{table}[h!]
    \caption{Estimation of $\theta( \approx 0.091)$ for different $(n,m)$ set, the results calculated from 1000 simulations are summarized as $AE(RMSD), \, h, p$. } 
    \centering 
    {\footnotesize\begin{tabular}{| c c c c c |} 
    \hline\hline 
        & $m = 20001 \, (\Delta = 0.001)$&$m = 41 \, (\Delta = 0.5)$&$m = 21 \, (\Delta =1)$ &$m = 5 \, (\Delta =5)$ \\
	\hline 
	$\bm{g} \, \Sigma \, \bm{g}^\top$ &$0.0195$  & $0.0224$ & $0.0238$ & $0.0582$  \\ 
	\hline
	$n = 50$& 0.0896  (0.0214) &  0.0893 (0.0238) & 0.0886 (0.0265) & 0.0878 (0.0396)  \\ 
	& $1, \, 0.0105$ & $1,\,0.0054$  & $1, \, 3.3015\times10^{-6}$ & $1,\,0.039\mid \eta = 136$  \\ 
	\hline
	$n = 100$& 0.0906 (0.0152) &  0.0905 (0.0167) & 0.0900 (0.0189) & 0.0892 (0.0289)\\
	& $0, \, 0.2868$ & $0, \, 0.0638$  & $1, \, 1.6454\times10^{-4}$ & $1, \, 0.0131 \mid \eta = 288$ \\ 
	\hline
	$n = 500$ & 0.0909 (0.0065) & 0.0909 (0.0070) & 0.0908 (0.0076) & 0.0904 (0.0127)\\
	& $0, \, 0.2644$ & $0, \, 0.2016$ & $0,\,0.0783$ & $1,\, 4.4399 \times 10^{-4} \mid \eta = 833$ \\ 
	\hline
	$n = 1000$  &0.0909 (0.0046) &  0.0910 (0.0049) & 0.0910 (0.0053) & 0.0909 (0.0093)  \\  
	& $0, \, 0.3926$ & $0, \, 0.3927$  & $0, \, 0.2868$ & $0,\, 0.0534\mid \eta = 974$ \\ [1ex] 
	\hline 
	\end{tabular}} \label{tab:ndelta}
	\label{table:nonlin} 
\end{table}
{
We applied a one-sample Kolmogorov-Smirnov test to check whether the estimates in Figure \ref{fig:Boxm} (a) are from a given normal distribution. Specifically, we tested whether $\sqrt{n} \left(\hat{\theta}_n - \theta\right)/\bm{g} \Sigma \bm{g}^\top $ in the 1000 simulations come from a standard normal distribution, using function $kstest$ in Matlab. The function returns a test decision $h$ ($h = 1$ if the test rejects the null hypothesis at the 5\% significance level, or $0$ otherwise) for the null hypothesis that the data comes from a standard normal distribution and the $p$-value $p$ of the hypothesis test. Figure \ref{fig:hist} depicts the normalized histograms of the 1000 simulations for $n = 1000, m = 20001, 41, 21, 5$. The asymptotic distribution of $\sqrt{n} \left(\hat{\theta}_n - \theta \right)$ and the values of $h$ and $p$ are also denoted in the plot. The test results indicate that the normal approximation of Theorem \ref{theorem:ad} is very good for a sample size of $n=10^3$.}

The estimation accuracy is obviously a function of the number of days sampled, and as we saw in the previous section accurate estimation is guaranteed as $n\to\infty$. In some cases sampling the queue many times during the day may be costly and it is therefore of interest to explore how many sampling instances are required for a good estimator. To have an overview of estimates under different $(n,m)$ set, we calculate the estimates for $n = 1000, 500, 100, 50, \, m = 20001, 41, 21,5$, and summarize the results as {AE (RMSD) in the first line, and $h, p$ in the second line in Table \ref{tab:ndelta}.  The value of $\bm{g} \Sigma \bm{g}^\top$ is also presented, and $\eta$ is $1000$ by default, that is, if $\eta$ is not displayed, then all the {1000} simulations were successful. It can be seen that for $m = 5$, $\eta$ increases with $n$ with almost instants successful for $n=1,000$. This further implies that for large enough sample sizes a good estimation can be obtained even for a small number of sampling times throughout the day.} 
The estimator quality is relatively robust to the number of observations each day. However, when the sample size is small it is better to sample the queue more often during the day to ensure an accurate estimator is obtained. {Also, when $m = 41, 20001$, the asymptotic normality of our estimator can already be observed when $n = 100$ (as was also illustrated in Figure~ \ref{fig:Boxm} (b)). The main conclusion from the simulation experiments is that 
a few observations of the queue length every day are sufficient for successful and accurate estimation of $\theta$. 
}

{
Finally, in practice,  if sampling the queue is costly, we do not need to observe the queue length for $m$ times each day for every day. For example, we can observe it for several days with a large $m$ number of points, in order to estimate $t_a$ and $t_b$ first. Then, we only need to select more than one point (less than $m$) in $\{0\} \cup [t_{\hat{a}_n}, t_{\hat{b}_n}]$ to estimate $\theta$. On the other hand, if the queue is only observable for a small number of days (i.e., small $n$), then it is better to sample more time instants within the day in order to increase the chance of successfully identifying times inside the support.
}

\section{Extensions} \label{sec:ext}

The estimator in Section \ref{subsec:E} can be easily modified to apply to the case with service closing time or the case where there are arrivals before time $0$. For the model with service closing time, we only discuss the case where there are no arrivals before time $0$. For other cases with service closing time, the analysis is similar. 

If $T \geq t_b$, the equilibrium arrival distribution is the same as the case service closing time, so the estimator is the same. When $t_a \leq T < t_b$, although the equilibrium arrival distribution is different, the relationship in Equation \eqref{eq:theta0} still holds, thus the estimator is exactly the same as $\hat{\theta}_n$ proposed in Section \ref{section:Estimator}. The only difference is that $\hat{b}_n \leq T$. With Algorithms \ref{alg:ql} and \ref{alg:ct} given in the appendix, we can calculate the equilibrium arrival distribution $F_e$ and the expected queue length $q(t)$. Both $F_e$ and $q(t)$ are plotted in Figure \ref{fig:2qs} (a). It can be observed that the expected queue length is a straight line from $t_a$ to $T$, and the slope can be used to estimate $\theta$. The estimator in \eqref{eq:theta2} works for the model with service closing time.

\begin{figure}[ht] 
	\subcaptionbox{The model with no arrivals before time $0$ but with service closing time ($\lambda = 5, \mu = 1,\alpha = 2, \beta = 0.2, T = 10$).}%
	{\includegraphics[width=0.495\linewidth]{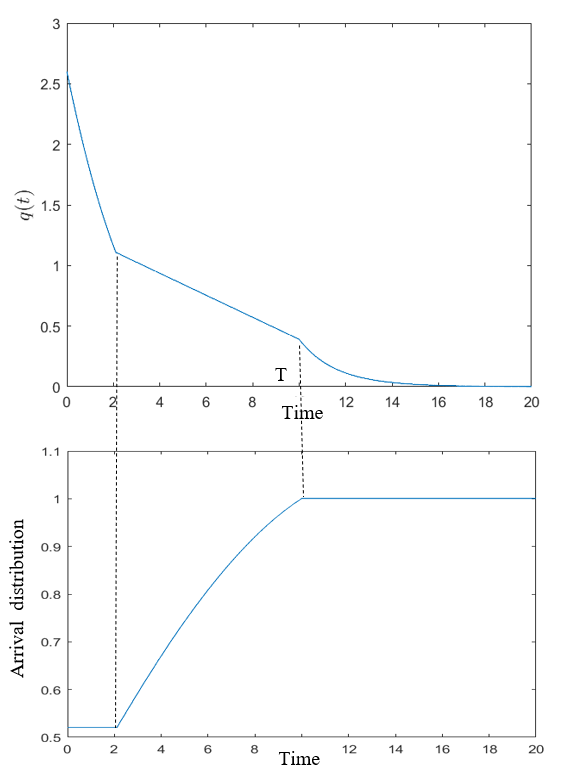}} 
	\subcaptionbox{The model with arrivals before time $0$ but no service closing time ($\lambda = 10, \mu = 1,\alpha = 2, \beta = 0.1$).}%
	{\includegraphics[width=0.508\linewidth]{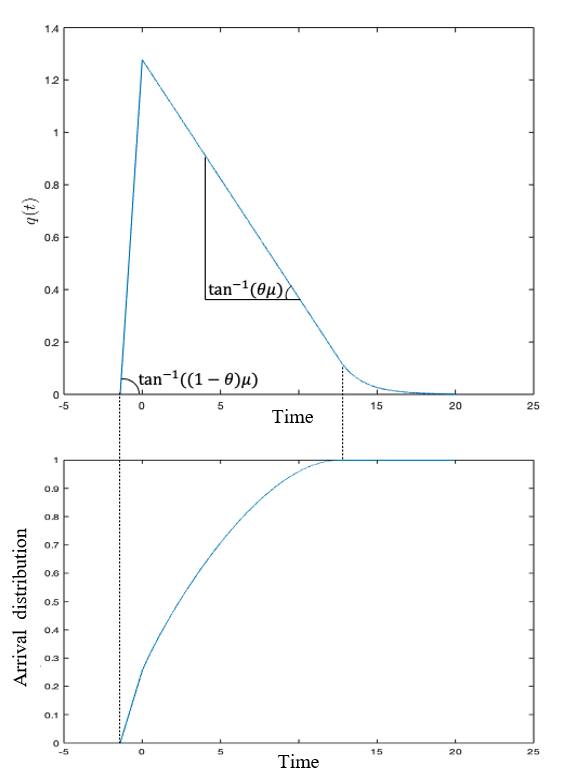}}
	\caption{The Nash equilibrium arrival distribution and the expected queue length under it.} \label{fig:2qs}
\end{figure}

In the model where there are arrivals before time $0$ and no service closing time, the expected cost can be written as
\begin{equation} 
	\mathbb{E}[C(t)]=
	\begin{cases}
		\displaystyle\frac{\alpha+\beta}{\mu}q(t) -\alpha t  & t < 0, \, \\[+6pt]
		\displaystyle\frac{\alpha+\beta}{\mu}q(t)+\beta t   & t \geq 0 \,.
	\end{cases}
\end{equation}
Since the support of the arrival distribution is $[-w,t_w]$, and the expected cost faced by a customer if she arrives at any $t \in [-w,t_w]$ is the same, we have
for $-w \leq s,~t \leq 0$, 
\begin{align}
    & -\alpha s +  (\alpha+\beta)\frac{q(s)}{\mu} \, = \, 	-\alpha t +  (\alpha+\beta)\frac{q(t)}{\mu} \\
    & \Leftrightarrow   \, \, \frac{\alpha}{\alpha+\beta} = \frac{q(s) - q(t)}{\mu \,(s-t)} \notag \\
    & \Leftrightarrow   \, \, \theta = 1-\frac{q(s) - q(t)}{\mu \,(s-t)} \,, \notag
\end{align}
for $-w \leq s \leq 0 < t \leq t_w$,
\begin{align}
	& -\alpha s +  (\alpha+\beta)\frac{q(s)}{\mu} \, = \,  \beta t + (\alpha+\beta)\frac{q(t)}{\mu} \label{eq:we1}\\
	& \Leftrightarrow  \, \, (\alpha+\beta)\frac{q(s) - q(t)}{\mu} = (\alpha+\beta)s - \beta(s-t)   \notag \\
    & \Leftrightarrow  \, \, \theta = \frac{s}{s-t} - \frac{q(s)-q(t)}{\mu \, (s-t)} \,,\notag
\end{align}
for $0 \leq s < t \leq t_w$,
\begin{align}
    & \beta s + (\alpha+\beta)\frac{q(s)}{\mu} \, = \, \beta t + (\alpha+\beta)\frac{q(t)}{\mu} \\
	& \Leftrightarrow  \,\, \theta = - \frac{q(s)-q(t)}{\mu \, (s-t)} \,.
\end{align}
Thus,
\begin{equation} \label{eq:thetaw}
	\theta = 
	\begin{dcases} 
		\displaystyle 1-\frac{q(s)-q(t)}{\mu (s-t)} & -w \leq s,~t \leq 0 \\
		\displaystyle \frac{s}{s-t} - \frac{q(s) -q(t)}{\mu(s-t)} &  -w \leq s \leq 0 < t \leq t_w \\
		\displaystyle  - \frac{q(s) -q(t)}{\mu(s-t)} & 0 \leq s < t \leq t_w \,.
	\end{dcases}
\end{equation}
That is $\theta$ can be expressed as a function of the expected queue length at $t \in [-w,t_w]$. Thus following the same estimation method as in Section \ref{section:Estimator}, the estimator of $\theta$ can be constructed based on the sample mean of the observed queue lengths in an estimated support of $[-w, t_w]$. The equilibrium arrival distribution $F_e$ can be calculated using Algorithm \ref{alg:w}. The expected queue length $q(t)$ can then be calculated using  \ref{alg:ql}. Both $F_e$ and $q(t)$ are plotted in Figure \ref{fig:2qs} (b). It can be inferred from Equation \eqref{eq:thetaw} that the expected queue length when {$t \in [-w,0]$ and $t \in [0, t_w]$} should form straight lines with slope $(1-\theta)\mu$ and $-\theta\mu$, respectively. This can also be observed in Figure \ref{fig:2qs} (b). 
\begin{remark}
Although $f_e$ is not continuous at $t = 0$, Equation \eqref{eq:we1} still holds. Thus, it is still possible to estimate $\theta$ when the two observation points inside the support are of opposite signs.
\end{remark}

Our methodology can be extended almost directly to other model variations, such as the model with order penalties in \citet{R14} and the model with earliness costs in \citet{SK17}.  We omit the details here. Furthermore, the method presented here has potential to be applied for systems with more elaborate dynamics. For example,  different service regimes such as processor sharing, or non-Markovian systems such as a G/G/1 with a general distribution for the number of customers and service times. In such cases however, the queue length observations are not sufficient and one must be able to sample the virtual workload (or waiting times) at different times instants, for example by sending small probes to the system. If estimating the workload is possible then an estimation equation similar to ours can be constructed from the equilibrium condition that the expected cost is constant throughout the support of the equilibrium arrival distribution. Note that computing the equilibrium arrival distribution in elaborate systems is typically intractable, but nevertheless the cost parameters can be estimated as long as the components of the cost function can be observed.

\section*{Acknowledgments}
\noindent  The authors would like to thank Peter Taylor for his valuable comments and advice. J. Wang would like to thank the University of Melbourne for supporting her work through the Melbourne Research Scholarship and the Albert Shimmins Fund.


\clearpage

\begin{appendices}
\section{Discrete approximation} \label{subsec:da}

This section explains how $F_e$ can be numerically approximated. 
If we denote a general arrival distribution by $F$ and its probability density function by $f$, then the arrival process is a non-homogeneous Poisson process with intensity measure $\lambda f(t)$ for all $t\geq 0$. The queue length dynamics satisfy the Kolmogorov forward equations
\begin{align}
	& P'_0(t) \, = \, P_1(t) \mu - P_0(t) \lambda f(t), \label{eq:D1} \\
	& P'_k(t) \, = \, P_{k-1}(t) \lambda f(t) + P_{k+1}(t) \mu - P_k(t) \label{eq:D2}\, \left(\lambda f(t)+\mu\right), \qquad k = 1,2,\ldots\,.
\end{align}
The equilibrium arrival distribution $F_e$ satisfies \eqref{eq:wof1}, \eqref{eq:ed2}, and a set of non-linear differential equations, which do not admit an analytic expression. We adopt the finite difference method, which was also mentioned in \citet[Section 3.1, Algorithm 1]{HR20} and was termed as a {\it discrete approximation}, to numerically obtain $F_e$ and the associated expected cost. 

To make the calculation of the expected queue length feasible, we truncate the queue length at $K$. Specifically, we assume that customers can choose to arrive at a time on a discrete grid $\mathcal{T} \equiv \{0, \delta, 2\delta, \ldots\}$, and the queue has a buffer size of $K$. When the value of $\delta$ is very small, with high probability there is at most one event happening in $\delta$, thus for $r = 1,2,\ldots$ and $k = 0,1, \ldots, K$, the queue length dynamics $P_k$ on $\mathcal{T}$ satisfy
{\footnotesize \begin{align} 
	& P_0((r+1)\delta) \, \approx \, P_0(r \delta) + P_1(r\delta) \mu -P_0(k\delta)\lambda f(r\delta) + o(\delta) \label{eq:Qdynamics1} \\
	& P_k((r+1)\delta) \, \approx \label{eq:Qdynamics2}\\ 
	& \qquad P_k(r\delta) + P_{k-1}(r\delta) \lambda f(r\delta) + P_{k+1}(r\delta) \mu - P_k(r\delta) \, \left(\lambda f(r\delta)+\mu\right) + o(\delta), \, 1 \leq k \leq K-1   \notag \\
	& P_K((r+1)\delta) \approx 1 - \sum_{k=0}^{K-1} \, P_k((r+1)\delta)+ o(\delta) \label{eq:Qdynamics3} \,, 
\end{align}}which are the finite difference scheme applied to Equations \eqref{eq:D1} and \eqref{eq:D2}. 
The expected cost 
\begin{equation} \label{eq:cost}
	\mathbb{E}_F[C(r\delta)]\, \approx \,
	\begin{dcases}
		\frac{(\alpha+\beta)\lambda p_e}{2\mu} & r = 0\\
		\frac{\alpha+\beta}{\mu} q(r\delta)+ \beta r \delta & r> 0 \,,
	\end{dcases}
\end{equation}
where $q(r\delta) \equiv \sum_{k=1}^{K}k P_k(r\delta)$ is the approximated expected queue length at slot $r$. For convenience, we drop the subscript $F$, and let the expected value and dynamics be that under the given arrival distribution for the rest of the paper. Increasing $K$ or decreasing $\delta$ clearly improves the accuracy of the approximation, but this is at the expense of calculation speed.  We set $K = \min \{m: \sum_{k = 0}^{m} {\lambda^k \, e^{-\lambda}}/k! \geq 1-10^{-6}\}$, and $\delta = 0.001$ throughout the paper.

The values of $t_a$ and $t_b$ are approximated by $r_a \delta$ and $r_b\delta$. In the following, we explain how to find $p_e$, $r_a$, $r_b$, and $f_e$ on $\mathcal{T}\cap[r_a\delta, r_b\delta]$. In each iteration, when the value of $p_e$ is given, the expected cost $\EX[C(0)]$ faced by customers arriving at time zero can be calculated. It follows from \cite{H13} that $F_e$ has a zero density along the interval $(0,t_a)$, which means $f(r \delta) = 0$ until $r  \geq r_a$. For $r = 1,2, \ldots, r_a$, since $f(r\delta) = 0$, the queue length dynamics at time $r\delta$ can be calculated using Equations \eqref{eq:Qdynamics1}-\eqref{eq:Qdynamics3}, the expected cost faced by a customer arriving at $r\delta$ can then be determined. The reason for $f_e(t) = 0, t \in (0,t_a)$ is that the expected cost faced by customers arriving at anytime in $(0,r_a \delta)$ is greater than $\EX[C(0)]$, which can also be inferred from Equation \eqref{eq:cost2}. Hence, to determine the value of $r_a$, we keep computing the queue dynamics, and then the expected cost for $t = r \delta$ from $r =1$ until $\EX[C(t)] \leq \EX[C(0)]$, then $r_a = \inf\{r: \EX[C(r \delta)] \leq \EX[C(0)], r \geq 1\}$. In \citet{H13}, the author calculated $t_a$ by working out the expression of the expected cost at time $t \in (0,t_a)$. Here we use an alternative way, and provide a more detailed explanation of the method in \citet{H13} and its comparison with our method in Remark~\ref{rem:DA} below. 

For $r \geq r_a$, the arrival density $f(r\delta)$ is defined by Equation \eqref{eq:wof1}, then the queue length dynamics can be obtained using Equations \eqref{eq:Qdynamics1}-\eqref{eq:Qdynamics3}. We keep calculating $f(r\delta)$ and the queue length dynamics until $f(r\delta) \leq 0$, and $r_b = \inf \{r: f(r\delta) \leq 0, r > r_a \}$. Thus, given the value of $p_e$, the values of $r_a$, $r_b$, and $f(r\delta)$ in $[r_a\delta, r_b \delta]$ can be determined. Another condition that $p_e$, $f_e$, $r_a$ and $r_b$ need to satisfy is 
\begin{equation}
	p_e+ \int_{t = t_a}^{t_b} f_e(t)\, dt = 1\,. \label{eq:One}
\end{equation}
Hence, we can initially guess a value for $p_e$, and then adjust it iteratively using the bisection method until Equation \eqref{eq:One} is satisfied. Specifically, we start with $p_1 = 0, p_2 = 1$, and always set $p_e=\displaystyle \frac{p_1+p_2}{2}$. At the end of each iteration, we set $p_2 = p_e$ if the total probability is greater than one, and $p_1 = p_e$ otherwise. This calculation process is summarized in Algorithm \ref{alg:wo}.

\begin{remark}\label{rem:DA}
The arrival distribution has zero density in $(0,t_a)$, so given $p_e$ at time zero, the expected waiting time faced by a customer if she arrives at any time $t \leq t_a$ has an analytic expression. This expression was derived in \citet[Lemma 3.3]{H13}, where the author proposed two methods to calculate its quantity. One method is computing it with the assistance of Bessel’s functions, and the other method is estimating it using a Monte Carlo simulation procedure. The goal of working out the expression is to find the time at which if a customer arrives, her expected cost will be the same as the expected cost if she arrives at time $0$. In our work, we do not adopt the expression of $t_a$, but keep calculating the expected cost until it is no longer greater than $\mathbb{E}[C(0)]$ and note down the time $\inf\{r: \EX[C(r \delta)] \leq \EX[C(0)], r \geq 1\}$. Although our method to estimate $t_a$ does not use the analytical properties of $t_a$, it performs very well. In fact, in all the numerical examples we tried, it calculated $t_a$ faster.
\end{remark}

\algnewcommand\algorithmicswitch{\textbf{switch}}
\algnewcommand\algorithmiccase{\textbf{case}}
\algnewcommand\algorithmicassert{\texttt{assert}}
\algnewcommand\Assert[1]{\State \algorithmicassert(#1)}%
\algdef{SE}[SWITCH]{Switch}{EndSwitch}[1]{\algorithmicswitch\ #1\ \algorithmicdo}{\algorithmicend\ \algorithmicswitch}%
\algdef{SE}[CASE]{Case}{EndCase}[1]{\algorithmiccase\ #1}{\algorithmicend\ \algorithmiccase}%
\algtext*{EndSwitch}%
\algtext*{EndCase}%
	
\section{Algorithms}
\subsection{The approximated expected queue length} \label{appendix:expectedQ}
	\begin{algorithm}[H] 
		\caption{Expected queue length.}		\label{alg:ql}
		\begin{algorithmic}[1]
		\scriptsize
		\State \textbf{Input:} $\lambda, \, \mu, \alpha, \, \beta, \, \delta, \, T$
		\State \textbf{Output:} $q(t)$
		\State $K = \min \{m: \sum_{k = 0}^{m} \lambda^k e^{-\lambda}/k! \geq 1-10^{-6}\}$
		\Switch{model} \Comment{Calculate the equilibrium arrival distribution whose details are in Algorithm \ref{alg:wo}, \ref{alg:ct}, and \ref{alg:w}}
		\Case{$1$} \Comment{The case where there are no arrivals before time $0$}
		\State $[p_e, t_a, t_b, f_e] = f_{\mathcal{W/O}}(\lambda, \mu, \alpha, \beta, \delta)$
		\EndCase
		\Case{$2$} \Comment{The case where there are no arrivals before time $0$, and the service closes at time $T$}
		\State $[p_e, t_a, t_b, f_e] = f_{\mathcal{CW/O}}(\lambda, \mu, \alpha, \beta, \delta, T)$
		\EndCase
		\Case{$3$} \Comment{The case where there are arrivals before time $0$}
		\State $[w, t_w, f_e] = f_{\mathcal{W}}(\lambda, \mu, \alpha, \beta, \delta)$
		\EndCase
		\EndSwitch
		\Procedure{queue length distribution under equilibrium after no one joins}{}
		\While{$P_0(r \delta) < 1$} \Comment{Calculate the queue dynamics until the system is empty}
		\State calculate $P_k(r \delta)$ for $k = 0,1,\ldots, K$  using $p_e, t_a, t_b, f_e$ or $w, t_w, f_e$, and Equations \eqref{eq:Qdynamics1}-\eqref{eq:Qdynamics3}
		\State$r = r+1$
		\EndWhile
		\EndProcedure
		\State $q(r\delta) = \sum_{k = 0}^{K} k \, P_k(r\delta)$ for $r = 0,1,2, \ldots$ \Comment{Approximate the expected queue lengths}
		\end{algorithmic}
	\end{algorithm}
	\subsection{The case with no arrivals before time $0$}\label{appendix:web}
	\begin{algorithm}[H]
		\caption{Equilibrium arrival distribution for the case with no arrivals before time $0$.}	\label{alg:wo}
		\begin{algorithmic}[1]	
			\scriptsize
			\State \textbf{Input:} $\lambda, \, \mu, \alpha, \, \beta, \, \delta$ 
			\State \textbf{Output:} equilibrium arrival distribution $F_e$, including $p_e$, $t_a$, $t_b$, and $f_e$ on $\{0\} \cup [t_a,t_b]$
			\Procedure{$f_{\mathcal{W/O}}$}{$\lambda, \mu, \alpha, \beta, \delta$} \Comment{Use bisection to obtain $p_e$}
			\State \textbf{Initialization:} $p_1 = 0, \, p_2 = 1, \, p_e=\displaystyle \frac{p_1+p_2}{2}$ \Comment{Initialize the value of $p_e$}
			\While{$p_2-p_1 > 10^{-6}$}  \Comment{Set the tolerance of the bisection method as $10^{-6}$}
			\State $c = \displaystyle \frac{(\alpha+\beta) \lambda p_e}{2 \mu}$ \Comment{The expected cost faced by a customer arriving at time $0$}
			\State $P_k(0) = \displaystyle \left(\frac{(\lambda p_e)^k \, e^{-\lambda p_e}}{k!}\right)_{k = 0:K}$ \Comment{The queue length distribution at time $0$}
			\State $r = 1$
			\Do
			\State calculate $P_k(r \delta)$ for $k = 0,1,\ldots, K$ using Equations \eqref{eq:Qdynamics1}-\eqref{eq:Qdynamics3}
			\If{$\left(\beta \, r \, \delta+(\alpha+\beta) \, \sum_{k = 0}^{K} \displaystyle \frac{ k \, P_k(r\delta)}{\mu} \right) \, > \,c$}  \Comment{Check whether $r \delta < t_a$}
			\State $f_e(r\delta) = 0$ \Comment{Arrival distribution density when $r \delta < t_a$}
			\Else
			\State$f_e(r\delta) = \displaystyle \frac{(1-P_0(r \delta))\mu}{\lambda}-\frac{\beta \mu}{(\alpha+\beta)\lambda}$ \Comment{Arrival distribution density when $r \delta \geq t_a$}
			\EndIf
			\State $r = r+1$
			\State $F_e = p_e+ \delta \,f_e \bm{e}$, where $\bm{e}$ is a vector of $1$'s of the appropriate size
			\doWhile{$F_e < 1$}
			\Comment{The calculation of $f_e$ stops if $F_e \geq 1$ or $f_e(t)< 0$}
			\If{$F_e >= 1$}
			\Comment{Update the range for $p_e$}
			\State $p_2 = p_e$ 
			\Else 
			\State $p_1=p_e$
			\EndIf
			\State $p_e = \displaystyle \frac{p_1+p_2}{2}$
			\EndWhile
			\State $t_a = \inf\{r\delta \, : \, f_e(r\delta) > 0\}$, $t_b  = \sup\{r\delta \, : \, f_e(r\delta) > 0\}$
			\EndProcedure
		\end{algorithmic}
	\end{algorithm}	
\subsection{Finite closing time} \label{appendix:extensions}
	\begin{algorithm}[H]
		\scriptsize
		\caption{Equilibrium arrival distribution for the model with service closing time.}	\label{alg:ct}
		\begin{algorithmic}[1]	
			\State \textbf{Input:} $\lambda, \, \mu, \alpha, \, \beta, \, \delta, \, T$ 
			\State \textbf{Output:} equilibrium arrival distribution $F_e$, including $p_e$, $t_a$, $t_b$, and $f_e$ on $\{0\} \cup [t_a,t_b]$ \Comment{$t_b$ can be $T$}
			\Procedure{$f_{\mathcal{CW/O}}$}{$\lambda, \mu, \alpha, \beta, \delta, T$} \Comment{Use Bisection to obtain $p_e$}
			\State \textbf{Initialization:} $p_1 = 0, \, p_2 = 1, \, p_e=\displaystyle \frac{p_1+p_2}{2}$ \Comment{Initialize the value of $p_e$}
			\While{$p_2-p_1 > 10^{-6}$} \Comment{Set the tolerance of the bisection method as $10^{-6}$}
			\State $c = \displaystyle \frac{(\alpha+\beta) \lambda p_e}{2 \mu}$ \Comment{The expected cost faced by a customer arriving at time $0$}
			\State $P_k(0) = \displaystyle \left(\frac{(\lambda p_e)^k \, e^{-\lambda p_e}}{k!}\right)_{k = 0:K}$ \Comment{The queue length distribution at time $0$}
			\For{$r = 1: \delta : \lceil \frac{T}{\delta} \rceil$}
			\State calculate $P_k(r \delta)$ for $k = 0,1,\ldots, K$ using Equations \eqref{eq:Qdynamics1}-\eqref{eq:Qdynamics3}
			\If{$\left(\beta \, (r \delta)+(\alpha+\beta) \, \sum_{k = 0}^{K} \displaystyle \frac{ k \, P_k(r\delta)}{\mu} \right) \, > \,c$}  \Comment{Check whether $r \delta < t_a$}
			\State $f_e(r\delta) = 0$ \Comment{Arrival distribution density when $r \delta < t_a$}
			\Else
			\State$f_e(r\delta) = \displaystyle \frac{(1-P_0(r \delta))\mu}{\lambda}-\frac{\beta \mu}{(\alpha+\beta)\lambda}$ \Comment{Arrival distribution density when $r \delta \geq t_a$}
			\EndIf 
			\If{$F_e >= 1$ or $f_e(r\delta) < 0$} 
			\State  break \Comment{The calculation of $f_e$ stops if $F_e \geq 1$ or $f_e(t)< 0$}
			\EndIf 
			\State $F_e = p_e+ \delta \,f_e \bm{e}$ 
			\EndFor
			\If{$F_e >= 1$} $\quad p_2 = p_e$  	\Comment{Update the range for $p_e$}
			\Else $\quad p_1=p_e$ 
			\EndIf 
			\State $p_e = \displaystyle \frac{p_1+p_2}{2}$
			\EndWhile 
			\State $t_a = \inf\{r\delta \, : \, f_e(r\delta) > 0\}$, $t_b  = \sup\{r\delta \, : \, f_e(r\delta) > 0\}$
			\EndProcedure
		\end{algorithmic}
	\end{algorithm}
	\subsection{The case with arrivals before time $0$}
	\begin{algorithm}[ht]
		\caption{Equilibrium arrival distribution for the case with arrivals before time $0$.}	\label{alg:w}
		\begin{algorithmic}[1]	
			\scriptsize
			\State \textbf{Input:} $\lambda, \, \mu, \alpha, \, \beta, \, \delta$ 
			\State \textbf{Output:} equilibrium arrival distribution $F_e$, including $w$, $t_w$, and $f_e$ on $[-w, t_w]$
			\Procedure{$f_{\mathcal{W}}$}{$\lambda, \mu, \alpha, \beta, \delta$}\Comment{Use Bisection to obtain $w$}
			\State \textbf{Initialization:} $t_1 = 0, \, t_2 =\displaystyle \frac{\lambda(\alpha+\beta)}{\mu\alpha}, \, w=\displaystyle \frac{t_1+t_2}{2}$ \Comment{Since $f_e(t) = \displaystyle\frac{\mu\alpha}{\lambda(\alpha+\beta)}$ for $t \in [-w,0]$, $0<w < \displaystyle \frac{\lambda(\alpha+\beta)}{\mu\alpha}$.}
			\While{$t_2-t_1 > 10^{-6}$} \Comment{Set the tolerance of the bisection method as $10^{-6}$}
			\State $c = \alpha w$ \Comment{The expected cost faced by a customer arriving at time $0$}
		    \State $p =\displaystyle w \, \frac{\mu \, \alpha}{\lambda (\alpha+\beta)}, \,~ P_k(r \delta) = \left(\displaystyle \frac{\left(\lambda p\right)^k \, e^{-\lambda p}}{k!}\right)_{k = 0:K}$ \Comment{The queue length distribution at time $0$}
			\State $r=1$
			\Do
			\State calculate $P_k(r \delta)$ for $k = 0,1,\ldots, K$ using Equations \eqref{eq:Qdynamics1}-\eqref{eq:Qdynamics3}
			\State$f_e(r\delta) = \displaystyle \frac{(1-P_0(r \delta))\mu}{\lambda}-\frac{\beta \mu}{(\alpha+\beta)\lambda}$ \Comment{Arrival distribution density when $-w \leq r \delta \leq t_w$}
			\State $r = r+1$;
			\State $F_e = p+ \delta \,f_e \bm{e}$, where $\bm{e}$ is a vector of $1$'s of the appropriate size.
			\doWhile{$F_e < 1$ and $\min f \geq 0$} \Comment{The calculation of $f_e$ stops if $F_e \geq 1$ or $f_e(t)< 0$}
			\If{$F_e >= 1$} $\quad t_2 = w$  	\Comment{Update the range for $w$}
			\Else $\quad t_1 = w$ 
			\EndIf
			\State $w = \displaystyle\frac{t_1+t_2}{2}$
			\EndWhile
			\State $t_w  = \sup\{r\delta \, : \, r \geq 1, f_e(r\delta) > 0\}$
			\EndProcedure
		\end{algorithmic}
	\end{algorithm}
\end{appendices}
\end{document}